\documentclass[a4paper,12pt]{article}
\usepackage[cp1251]{inputenc}
\usepackage[russian]{babel}
\usepackage{amsfonts, amssymb, amsmath, amsthm, amscd}
\usepackage{cite}
\author{A.A. Vasil'eva}
\title{Widths of function classes on sets with tree-like
structure\footnote{The research was carried
out with the financial support of the Russian Foundation for Basic
Research (grants no. 13-01-00022, 12-01-00554)}}
\date{}
\begin{document}

\maketitle

\newenvironment{Biblio}{%
                  \renewcommand{\refname}{\footnotesize REFERENCES}%
                  \begin{thebibliography}{99}%
                  \itemsep -0.2em}%
                  {\end{thebibliography}}

\def\inff{\mathop{\smash\inf\vphantom\sup}}
\renewcommand{\le}{\leqslant}
\renewcommand{\ge}{\geqslant}
\newcommand{\sgn}{\mathrm {sgn}\,}
\newcommand{\inter}{\mathrm {int}\,}
\newcommand{\dist}{\mathrm {dist}}
\newcommand{\supp}{\mathrm {supp}\,}
\newcommand{\R}{\mathbb{R}}
\renewcommand{\C}{\mathbb{C}}
\newcommand{\Z}{\mathbb{Z}}
\newcommand{\N}{\mathbb{N}}
\newcommand{\Q}{\mathbb{Q}}
\theoremstyle{plain}
\newtheorem{Trm}{Theorem}
\newtheorem{trma}{Theorem}

\newtheorem{Def}{Definition}
\newtheorem{Sup}{Assumption}
\newtheorem{Cor}{Corollary}
\newtheorem{Lem}{Lemma}
\newtheorem{Rem}{Remark}
\newtheorem{Sta}{Proposition}
\renewcommand{\proofname}{\bf Proof}
\renewcommand{\thetrma}{\Alph{trma}}

\section{Introduction}

In this paper, estimates for Kolmogorov, Gelfand and linear widths
of functional classes on sets with a tree-like structure are
obtained. As examples we consider weighted Sobolev classes on a
John domain, as well as some function classes on a metric and
combinatorial tree.

Let $\Omega \subset \R^d$ be a bounded domain (an open connected
set), and let $g$, $v:\Omega\rightarrow \R_+$ be measurable
functions. For each measurable vector-valued function $\psi:\
\Omega\rightarrow \R^l$, $\psi=(\psi_k) _{1\le k\le l}$, and for
each $p\in [1, \, \infty]$, we put
$$
\|\psi\|_{L_p(\Omega)}= \Big\|\max _{1\le k\le l}|\psi _k |
\Big\|_p=\left(\int \limits_\Omega \max _{1\le k\le l}|\psi _k(x)
|^p\, dx\right)^{1/p}
$$
(appropriately modified if $p=\infty$). Let
$\overline{\beta}=(\beta _1, \, \dots, \, \beta _d)\in
\Z_+^d:=(\N\cup\{0\})^d$, $|\overline{\beta}| =\beta _1+
\ldots+\beta _d$. For any distribution $f$ defined on $\Omega$ we
write $\displaystyle \nabla ^r\!f=\left(\partial^{r}\! f/\partial
x^{\overline{\beta}}\right)_{|\overline{\beta}| =r}$ (here partial
derivatives are taken in the sense of distributions), and denote
by $l_{r,d}$ the number of components of the vector-valued
distribution $\nabla ^r\!f$. We set
$$
W^r_{p,g}(\Omega)=\left\{f:\ \Omega\rightarrow \R\big| \; \exists
\psi :\ \Omega\rightarrow \R^{l_{r,d}}\!:\ \| \psi \|
_{L_p(\Omega)}\le 1, \, \nabla ^r\! f=g\cdot \psi\right\}
$$
\Big(we denote the corresponding function $\psi$ by
$\displaystyle\frac{\nabla ^r\!f}{g}$\Big),
$$
\| f\|_{L_{q,v}(\Omega)}{=}\| f\|_{q,v}{=}\|
fv\|_{L_q(\Omega)},\qquad L_{q,v}(\Omega)=\left\{f:\Omega
\rightarrow \R| \; \ \| f\| _{q,v}<\infty\right\}.
$$
We call the set $W^r_{p,g}(\Omega)$ a weighted Sobolev class.
Observe that $W^r_{p,1}(\Omega)=W^r_p(\Omega)$.

For properties of weighted Sobolev spaces and their
generalizations, see the books \cite{triebel, kufner,
edm_trieb_book, triebel1, edm_ev_book, myn_otel} and the survey
paper \cite{kudr_nik}. Sufficient conditions for boundedness and
compactness of embeddings of weighted Sobolev spaces into weighted
$L_q$-spaces were obtained by Kudryavtsev, Ne\v{c}as, Kufner,
Triebel, Gurka and Opic, Besov, Caso and D’Ambrosio, and other
authors (see, e.g., \cite{kudrjavcev, j_necas, kufner, triebel,
gur_opic1, besov4, ambr_l}). Note that, in these papers, the
definition of weighted Sobolev classes included conditions not
only on the derivatives of order $r$ but also on derivatives of
lesser order. For more details, see \cite{vas_vl_raspr}.

Let $X$, $Y$ be sets, $f_1$, $f_2:\ X\times Y\rightarrow \R_+$. We
write $f_1(x, \, y)\underset{y}{\lesssim} f_2(x, \, y)$ (or
$f_2(x, \, y)\underset{y}{\gtrsim} f_1(x, \, y)$) if, for any
$y\in Y$, there exists $c(y)>0$ such that $f_1(x, \, y)\le
c(y)f_2(x, \, y)$ for each $x\in X$; $f_1(x, \,
y)\underset{y}{\asymp} f_2(x, \, y)$ if $f_1(x, \, y)
\underset{y}{\lesssim} f_2(x, \, y)$ and $f_2(x, \,
y)\underset{y}{\lesssim} f_1(x, \, y)$.

For $x\in \R^d$ and $a>0$ we shall denote by  $B_a(x)$ the closed
Euclidean ball of radius $a$ in $\R^d$ centered at the point $x$.

Let $|\cdot|$ be an arbitrary norm on $\R^d$, and let $E, \,
E'\subset \R^d$, $x\in \R^d$. We put
$$
{\rm diam}_{|\cdot|}\, E=\sup \{|y-z|:\; y, \, z\in E\}, \;\; {\rm
dist}_{|\cdot|}\, (x, \, E)=\inf \{|x-y|:\; y\in E\},
$$
$$
{\rm dist}_{|\cdot|}\, (E', \, E)=\inf \{|x-y|:\; x\in E, \, y\in
E'\}.
$$

Denote by $AC[t_0, \, t_1]$ the space of absolutely continuous
functions on an interval $[t_0, \, t_1]$.
\begin{Def}
\label{fca} Let $\Omega\subset\R^d$ be a bounded domain, and let
$a>0$. We say that $\Omega \in {\bf FC}(a)$ if there exists a
point $x_*\in \Omega$ such that, for any $x\in \Omega$, there
exists a curve $\gamma _x:[0, \, T(x)] \rightarrow\Omega$ with the
following properties:
\begin{enumerate}
\item $\gamma _x\in AC[0, \, T(x)]$, $\left|\frac{d\gamma _x(t)}{dt}\right|=1$ a.e.,
\item $\gamma _x(0)=x$, $\gamma _x(T(x))=x_*$,
\item $B_{at}(\gamma _x(t))\subset \Omega$ for any $t\in [0, \, T(x)]$.
\end{enumerate}
\end{Def}

\begin{Def}
We say that $\Omega$ satisfies the John condition (and call
$\Omega$ a John domain) if $\Omega\in {\bf FC}(a)$ for some $a>0$.
\end{Def}

\begin{Rem}
\label{rem2} If $\Omega\in {\bf FC}(a)$ and a point $x_*$ is such
as in Definition {\rm \ref{fca}}, then
\begin{align}
\label{diam_dist} {\rm diam}_{|\cdot|}\, \Omega
\underset{d,a,|\cdot|}{\lesssim} {\rm dist}_{|\cdot|}\, (x_*, \,
\partial \Omega).
\end{align}
\end{Rem}

Let $(X, \, \|\cdot\|_X)$ be a normed space, let $X^*$ be its
dual, and let ${\cal L}_n(X)$, $n\in \Z_+$, be the family of
subspaces of $X$ of dimension at most $n$. Denote by $L(X, \, Y)$
the space of continuous linear operators from $X$ into a normed
space $Y$. Also, by ${\rm rk}\, A$ denote the dimension of the
image of an operator $A\in L(X, \, Y)$, and by $\| A\|
_{X\rightarrow Y}$, its norm.

By the Kolmogorov $n$-width of a set $M\subset X$ in the space
$X$, we mean the quantity
$$
d_n(M, \, X)=\inff _{L\in {\cal L}_n(X)} \sup_{x\in M}\inff_{y\in
L}\|x-y\|_X,
$$
by the linear $n$-width, the quantity
$$
\lambda_n(M, \, X) =\inff_{A\in L(X, \, X), \, {\rm rk} A\le
n}\sup _{x\in M}\| x-Ax\| _X,
$$
and by the Gelfand $n$-width, the quantity
$$
d^n(M, \, X)=\inff _{x_1^*, \, \dots, \, x_n^*\in X^*} \sup
\{\|x\|:\; x\in M, \, x^*_j(x)=0, \; 1\le j\le n\}=
$$
$$
=\inff _{A\in L(X, \, \R^n)}\sup \{\|x\|:\; x\in M\cap \ker A\}.
$$
In \cite{pietsch1} the definition of strict $s$-numbers of a
linear continuous operator was given. In particular, Kolmogorov
numbers of an operator $A:X\rightarrow Y$ coincide with Kolmogorov
widths $d_n(A(B_X), \, Y)$ (here $B_X$ is the unit ball in the
space $X$); if the operator is compact, then its approximation
numbers coincide with linear widths $\lambda_n(A(B_X), \, Y)$ (see
the paper of Heinrich \cite{heinr}). If $X$ and $Y$ are both
uniformly convex and uniformly smooth and $A:X\rightarrow Y$ is a
bounded linear map with trivial kernel and range dense in $Y$,
then Gelfand numbers of $A$ are equal to $d^n(A(B_X), \, Y)$ (see
the paper of Edmunds and Lang \cite{edm_lang1}).

In the 1960–1970s problems concerning the values of the widths of
function classes in $L_q$ and of finite-dimensional balls $B_p^n$
in $l_q^n$ were intensively studied. Here $l_q^n$ $(1\le q\le
\infty)$ is the space $\R^n$ with the norm
$$
\|(x_1, \, \dots , \, x_n)\| _q\equiv\|(x_1, \, \dots ,
\, x_n)\| _{l_q^n}= \left\{
\begin{array}{l}(| x_1 | ^q+\dots+ | x_n | ^q)^{1/q},\text{ if
}q<\infty ,\\ \max \{| x_1 | , \, \dots, \, | x_n |\},\text{ if
}q=\infty ,\end{array}\right .
$$
$B_p^n$ is the unit ball in $l_p^n$. For more details, see
\cite{tikh_nvtp, itogi_nt, kniga_pinkusa}.

Besov in \cite{besov_peak_width} proved the result on the
coincidence of orders of widths
$$
d_n(W^r_p(K_\sigma), \, L_q(K_\sigma))\underset{p,q,r,d,\sigma}
{\asymp} d_n(W^r_p([0, \, 1]^d), \, L_q([0, \, 1]^d));
$$
here $$K_\sigma=\{(x_1, \, \dots, \, x_{d-1}, \, x_d): \; |(x_1,
\, \dots, \, x_{d-1})|^{1/\sigma}<x_d<1\},$$ $\sigma>1$,
$r-[\sigma(d-1)+1]\left(\frac 1p-\frac 1q\right)_+>0$ and some
conditions on the parameters $p$, $q$, $r$, $d$ hold (see Theorem
\ref{sob_dn} below). For $r=1$, $p=q$ and more general ridged
domains, estimates of approximation numbers were obtained by W.D.
Evans and D.J. Harris \cite{evans_har}.

An upper estimate of Kolmogorov widths of Sobolev classes on a
cube in a weighted $L_p$-space was first obtained by Birman and
Solomyak \cite{birm} (for $q>\max \{p, \, 2\}$, the orders of this
bound are not sharp). In \cite{el_kolli}, El Kolli had found the
orders of the quantities $d_n(W^r_{p,g}(\Omega), \,
L_{q,v}(\Omega))$, where $\Omega$ is a bounded domain with smooth
boundary, $p=q$, and weight functions $g$ and $v$ were equal to a
power of the distance to the boundary of $\Omega$; by using Banach
space interpolation, Triebel \cite{triebel} extended the upper
bounds to the widths $d_n(W^r_{p,g}(\Omega), \, L_{q,v}(\Omega))$
for $p\le q$. Estimates for linear widths of weighted Sobolev
classes on~$\R^d$ with weights of the form
$w_\alpha(x)=(1+\|x\|_2^2)^{\alpha/2}$ were established by Mynbaev
and Otelbaev in \cite{myn_otel}. For intersections of some
weighted Sobolev classes on a cube with weights that are powers of
the distance from the boundary, order estimates of widths were
obtained by Boykov \cite{boy_1, boy_2}. In \cite{tr_jat} Triebel
obtained estimates of approximation numbers for weighted Sobolev
classes with weights that have singularity at a~point (case
$p=q$); this result was generalized in \cite{vas_sing}. For
general weights, the Kolmogorov’s and approximation numbers of an
embedding operator of Sobolev classes in $L_p$ were estimated by
Lizorkin, Otelbaev, Aitenova and Kusainova \cite{liz_otel1,
otelbaev, ait_kus1, ait_kus2}.

Let us formulate the main result of this paper.

Denote by $\mathbb{H}$ the set of all nondecreasing positive
functions defined on $(0, \, 1]$.

Introduce the notion of $h$-set in accordance with
\cite{m_bricchi1}.
\begin{Def}
\label{h_set} Let $\Gamma\subset \R^d$ be a nonempty compact set
and $h\in \mathbb{H}$. We say that $\Gamma$ is an $h$-set if there
are a $\hat c\ge 1$ and a finite countably additive measure $\mu$
on $\R^d$ such that $\supp \mu=\Gamma$ and $ \hat c^{-1}h(t)\le
\mu(B_t(x))\le \hat c h(t)$ for any $x\in \Gamma$ and $t\in (0, \,
1]$.
\end{Def}
Observe that the measure $\mu$ is nonnegative.

Below we consider a function $h\in \mathbb{H}$ which has the
following form in a neighborhood of the zero:
\begin{align}
\label{def_h} h(t)=t^{\theta}\Lambda(t), \;\;\; 0\le \theta<d,
\end{align}
where $\Lambda:(0, \, +\infty)\rightarrow (0, \, +\infty)$ is an
absolutely continuous function such that
\begin{align}
\label{yty} \frac{t\Lambda'(t)}{\Lambda(t)} \underset{t\to+0}{\to}
0.
\end{align}

Everywhere below, we use the notation $\log x=\log_2 x$.

\label{gamma_def}Let $\Omega\in {\bf FC}(a)$ be a bounded domain
and $\Gamma\subset \partial \Omega$ an $h$-set. Below we assume
for convenience that $\Omega\subset \left[-\frac 12, \, \frac
12\right]^d$ (the general case can be reduced to this special case
by a change of variables). Let $1<p\le \infty$, $1\le q<\infty$,
$r\in \N$, $\delta:=r+\frac dq-\frac dp>0$, $\beta_g$, $\beta_v\in
\R$, $g(x)=\varphi_g({\rm dist}_{|\cdot|}(x, \, \Gamma))$,
$v(x)=\varphi_v({\rm dist}_{|\cdot|}(x, \, \Gamma))$,
\begin{align}
\label{ghi_g0} \varphi_g(t)=t^{-\beta_g}\Psi_g(t), \;\;
\varphi_v(t)=t^{-\beta_v} \Psi_v(t),
\end{align}
and the functions $\Psi_g$ and $\Psi_v$ are absolutely continuous,
\begin{align}
\label{psi_cond} \frac{t\Psi'_g(t)}{\Psi_g(t)}
\underset{t\to+0}{\to}0, \;\; \frac{t\Psi'_v(t)}{\Psi_v(t)}
\underset{t\to+0}{\to}0,
\end{align}
and
\begin{align}
\label{muck} -\beta_vq+d-\theta>0.
\end{align}
Moreover, assume that
\begin{align}
\label{beta} \text{a) }\beta_g+\beta_v<\delta-\theta\left(\frac
1q-\frac 1p\right)_+ \quad\text{ or b) }\beta_g+\beta_v=\delta
-\theta\left(\frac 1q-\frac 1p\right)_+.
\end{align}
We also assume that
\begin{align}
\label{phi_g} \begin{array}{c} \Lambda(t)=|\log
t|^\gamma\tau(|\log t|), \;\; \Psi_g(t)=|\log t|^{-\alpha_g}
\rho_g(|\log t|),\\ \Psi_v(t)=|\log t|^{-\alpha_v} \rho_v(|\log
t|) \quad \text{in case b)}, \end{array}
\end{align}
the functions $\rho_g$, $\rho_v$ and $\tau$ are absolutely
continuous,
\begin{align}
\label{ll} \lim \limits _{y\to +\infty}\frac{y\tau'(y)}{\tau(y)}=
\lim \limits _{y\to +\infty}\frac{y\rho_g'(y)}{\rho_g(y)}=\lim
\limits _{y\to +\infty}\frac{y\rho_v'(y)}{\rho_v(y)}=0,
\end{align}
\begin{align}
\label{g0ag} \alpha:=\alpha_g+\alpha_v
>(1-\gamma)\left(\frac 1q-\frac 1p\right)_+; \quad
\gamma<0 \quad\text{if }\quad \theta=0.
\end{align}
It can readily be seen that, in this case, the functions
$\Lambda$, $\Psi_g$ and $\Psi_v$ satisfy (\ref{yty}) and
(\ref{psi_cond}), respectively.
\begin{Rem}
\label{ste} If functions $\Psi_g$ and $\Psi_v$ ($\rho_g$ and
$\rho_v$) satisfy (\ref{psi_cond}) ((\ref{ll}), respectively),
then their product and all powers of these functions satisfy a
similar condition.
\end{Rem}

Set $$\beta=\beta_g+\beta_v, \quad \rho(y)=\rho_g(y)\rho_v(y),
\quad \Psi(y)=\Psi_g(y)\Psi_v(y),$$ $\mathfrak{Z}=(r,\, d, \, p,
\, q, \, g, \, v, \, h, \, a, \, \hat c)$,
$\mathfrak{Z}_*=(\mathfrak{Z}, \, R)$, where $\hat c$ is the
constant used in Definition \ref{h_set} and \label{r_def}$R={\rm
diam}\, \Omega$.

Given $1\le p\le \infty$, we write $p'=\frac{p}{p-1}$.

In estimating Kolmogorov, linear, and Gelfand widths we set,
respectively, $\vartheta_l(M, \, X)=d_l(M, \, X)$ and $\hat q=q$,
$\vartheta_l(M, \, X)=\lambda_l(M, \, X)$ and $\hat q=\min\{q, \,
p'\}$, $\vartheta_l(M, \, X)=d^l(M, \, X)$ and $\hat q=p'$.

Denote $\psi_\Lambda(y)= \frac{1}{\Lambda\left(
\frac{1}{y}\right)}$. Then $\psi_\Lambda\in AC(0, \, \infty)$ and
$\lim \limits _{y\to\infty} \frac{y\psi_\Lambda'(y)}
{\psi_\Lambda(y)}=0$.

Let $\gamma_*>0$, $\psi_*\in AC(0, \, \infty)$, $\lim \limits
_{y\to\infty} \frac{y\psi'_*(y)} {\psi_*(y)}=0$. It will be proved
below (see Lemma \ref{obr}) that for sufficiently large $x$ the
function $y^{\gamma_*}\psi_*(y)$ is strictly increasing and the
equation
$$y^{\gamma_*}\psi_*(y)=x$$ has the unique solution $y(x)$.
Moreover,
$$
y(x)=x^{\frac{1}{\gamma_*}}\varphi_*(x),
$$
where $\varphi_*$ is an absolutely continuous function and $\lim
\limits _{x\to\infty} \frac{x\varphi_*'(x)} {\varphi_*(x)}=0$. We
denote the function $\varphi_*$ by \label{phi_lam}$\varphi
_{\gamma_*,\psi_*}$.
\begin{Trm}
\label{main_sobol} There exists $n_0=n_0(\mathfrak{Z})$ such that
for any $n\ge n_0$ the following assertions hold.
\begin{enumerate}
\item Suppose that $\theta>0$ and the condition (\ref{beta}), a) holds.
\begin{itemize}
\item Let $p\ge q$ or $p< q$, $\hat q\le 2$.
We set
\begin{align}
\label{case1_theta_j_0} \theta_1=\frac{\delta}{d}-\left(\frac
1q-\frac 1p\right)_+, \quad  \theta_2=\frac{\delta
-\beta}{\theta}-\left(\frac 1q-\frac 1p\right)_+,
\end{align}
\begin{align}
\label{case1_sigma_j_0} \sigma_1(n)=1, \quad
\sigma_2(n)=\Psi(n^{-1/\theta} \varphi_{\theta,
\psi_\Lambda}^{-1}(n) )\varphi_{\theta,
\psi_\Lambda}^{\beta-\delta}(n).
\end{align}
Let $\theta_1\ne \theta_2$, $j_*\in \{1, \, 2\}$,
$$
\theta_{j_*}=\min\{\theta_1, \, \theta_2\}.
$$
Then
$$
\vartheta_n(W^r_{p,g}(\Omega), \, L_{q,v}(\Omega))
\underset{\mathfrak{Z}_*}{\asymp} n^{-\theta_{j_*}}\sigma_{j_*}(n).
$$
\item Let $p<q$, $\hat q>2$. Denote
\begin{align}
\label{case1qg2_th12} \theta_1=\frac{\delta}{d} +
\min\left\{\frac 1p-\frac 1q, \, \frac 12-\frac{1}{\hat q}\right\},
\quad \theta_2=\frac{\hat q\delta}{2d},
\end{align}
\begin{align}
\label{case1qg2_th34}
\theta_3=\frac{\delta-\beta}{\theta} + \min\left\{\frac 1p-\frac 1q,
\, \frac 12-\frac{1}{\hat q}\right\}, \quad \theta_4 = \frac{\hat q(\delta
-\beta)}{2\theta},
\end{align}
\begin{align}
\label{case1def_sigma_qg2}
\begin{array}{c}
\displaystyle \sigma_1(n)=\sigma_2(n)=1, \quad \sigma_3(n) =
\Psi(n^{-1/\theta} \varphi_{\theta, \psi_\Lambda}^{-1}(n)
)\varphi_{\theta, \psi_\Lambda}^{\beta-\delta}(n),\\
\sigma_4(n)=\sigma_3(n^{\hat q/2}).
\end{array}
\end{align}
Suppose that there exists $j_*\in \{1, \, 2, \, 3, \, 4\}$ such
that
\begin{align}
\label{thj_min_case1}
\theta_{j_*}<\min_{j\ne j_*} \theta_j.
\end{align}
Then
$$
\vartheta_n(W^r_{p,g}(\Omega), \, L_{q,v}(\Omega))
\underset{\mathfrak{Z}_*}{\asymp} n^{-\theta_{j_*}}\sigma_{j_*}(n).
$$
\end{itemize}

\item Suppose that $\theta>0$ and the condition (\ref{beta}), b) holds.
Then
$$
\vartheta_n(W^r_{p,g}(\Omega), \, L_{q,v}(\Omega))
\underset{\mathfrak{Z}_*}{\asymp} (\log
n)^{-\alpha+(1-\gamma)\left(\frac 1q-\frac 1p\right)_+} \rho(\log
n)\tau^{-\left(\frac 1q-\frac 1p\right)_+}(\log n).
$$
\item Suppose that $\theta=0$ and the condition (\ref{beta}), b)
holds. Denote $\tau^{-1}(x)=\frac{1}{\tau(x)}$.
\begin{itemize}
\item Let $p\ge q$ or $p<q$, $\hat q\le 2$. We set
\begin{align}
\label{case3_pleq_theta_j} \theta _1=\frac{\delta}{d} -\left(\frac
1q-\frac 1p\right)_+, \quad \theta
_2=\frac{\alpha}{1-\gamma}-\left(\frac 1q-\frac 1p\right)_+,
\end{align}
\begin{align}
\label{case3_pleq_sigma} \sigma_1(n)=1, \quad
\sigma_2(n)=\rho\left(n^{\frac{1}{1-\gamma}} \varphi
_{1-\gamma,\tau^{-1}}(n)\right)
\varphi_{1-\gamma,\tau^{-1}}^{-\alpha}(n).
\end{align}
Let $\theta_1\ne \theta_2$, $j_*\in \{1, \, 2\}$,
$$
\theta_{j_*}=\min\{\theta_1, \, \theta_2\}.
$$
Then
$$
\vartheta_n(W^r_{p,g}(\Omega), \, L_{q,v}(\Omega))
\underset{\mathfrak{Z}_*}{\asymp} n^{-\theta_{j_*}}\sigma _{j_*}(n).
$$
\item Let $p<q$ and $\hat q>2$. Denote
\begin{align}
\label{case3_qg2_th12}
\theta_1=\frac{\delta}{d} +
\min\left\{\frac 1p-\frac 1q, \, \frac 12-\frac{1}{\hat q}\right\},
\quad \theta_2=\frac{\hat q\delta}{2d},
\end{align}
\begin{align}
\label{case3_qg2_th34}
\theta_3=\frac{\alpha}{1-\gamma} +\min\left\{\frac 1p-\frac 1q, \,
\frac 12-\frac{1}{\hat q}\right\}, \quad
\theta_4=\frac{\hat q\alpha}{2(1-\gamma)},
\end{align}
\begin{align}
\label{case3_qg2_sigma}
\begin{array}{c}
\displaystyle \sigma_1(n)=\sigma_2(n)=1, \quad \sigma_3(n)=
\rho\left(n^{\frac{1}{1-\gamma}}\varphi_{1-\gamma,\tau^{-1}}(n)\right)
\varphi_{1-\gamma,\tau^{-1}}^{-\alpha}(n), \\ \sigma_4(n)=
\sigma_3(n^{\hat q/2}).
\end{array}
\end{align}
Suppose that there exists $j_*\in \{1, \, 2, \, 3, \, 4\}$ such
that
\begin{align}
\label{case3_qg2_ineq}
\theta_{j_*}<\min _{j\ne j_*}\theta_j.
\end{align}
Then
$$
\vartheta_n(W^r_{p,g}(\Omega), \, L_{q,v}(\Omega))
\underset{\mathfrak{Z}_*}{\asymp} n^{-\theta_{j_*}}\sigma _{j_*}(n).
$$
\end{itemize}
\item Suppose that $\theta=0$ and the condition (\ref{beta}), a) holds.
\begin{itemize}
\item If $p\ge q$ or $p<q$, $\hat q\le 2$, then
$$
\vartheta_n(W^r_{p,g}(\Omega), \, L_{q,v}(\Omega)) \asymp
n^{-\frac{\delta}{d}+\left(\frac 1q-\frac 1p\right)_+}.
$$
\item Let $p<q$, $\hat q>2$, $\theta_1=\frac{\delta}{d} +
\min\left\{\frac 1p-\frac 1q, \, \frac 12-\frac{1}{\hat
q}\right\}$, $\theta_2=\frac{\hat q\delta}{2d}$, $\theta_1\ne
\theta_2$. Then
$$
\vartheta_n(W^r_{p,g}(\Omega), \, L_{q,v}(\Omega))
\underset{\mathfrak{Z}_*}{\asymp} n^{-\min\{\theta_1, \,
\theta_2\}}.
$$
\end{itemize}

\end{enumerate}
\end{Trm}

Notice that from the equivalence of norms on $\R^d$ and from the
definition of weights $g$ and $v$ it follows that it is sufficient
to consider the norm $|x|=|(x_1, \, \dots, \, x_d)|=\max _{1\le
i\le d}|x_i|$ for $x\in \R^d$.

\section{Notation}
In what follows $\overline{A}$ (${\rm int}\, A$, $\partial A$,
${\rm card}\,A$, respectively) be, respectively, the closure
(interior, boundary, cardinality) of $A$. If a set $A$ is
contained in some subspace $L\subset \R^d$ of dimension  $(d-1)$,
then we denote by ${\rm int}_{d-1}A$ the interior of $A$ with
respect to the induced topology on the space $L$.  For a convex
set $A$ we denote by $\dim A$  the dimension of the affine span of
the set $A$.

Let $(\Omega, \, \Sigma, \, \mu)$ be a measure space. We say that
sets $A$, $B\subset \Omega$ do not overlap if $A\cap B$ is a
Lebesgue nullset.  Let $E$, $E_1, \, \dots, \, E_m\subset \Omega$
be measurable sets, $m\in \N\cup \{\infty\}$. We say that
$\{E_i\}_{i=1}^m$ is a partition of $E$ if the sets $E_i$ do not
overlap pairwise and the set $\left(\cup _{i=1}^m
E_i\right)\bigtriangleup E$ is a Lebesgue nullset.

Denote by $\chi_E(\cdot)$ the indicator function of a set $E$.

A set $A\subset \R^d$ is said to be a cube if there are $s_j<
t_j$, $1\le j\le d$, such that $t_j-s_j=t_1-s_1$ for any $j=1, \,
\dots , \, d$ and
$$
\prod _{j=1}^d (s_j, \, t_j)\subset A\subset\prod _{j=1}^d [s_j,
\, t_j].
$$

For every cube $K$ and $s\in \Z_+$, denote by $\Xi _s(K)$ the
partition of $K$ into $2^{sd}$ non-overlapping cubes of the same
size, $\Xi(K):=\bigcup_{s\in \Z_+} \Xi _s(K)$. Everywhere below,
we assume that these cubes are closed.

We recall some definitions from graph theory. Throughout, we
assume that the graphs have neither multiple edges nor loops.

Let ${\cal G}$ be a graph containing at most countable number of
vertices. We shall denote by ${\bf V}({\cal G})$ and by ${\bf
E}({\cal G})$ the set of vertices and the set of edges of ${\cal
G}$, respectively. Two vertices are called {\it adjacent} if there
is an edge between them. We shall identify pairs of adjacent
vertices with edges that connect them. Let $\xi_i\in {\bf V}({\cal
G})$, $1\le i\le n$. The sequence $(\xi_1, \, \dots, \, \xi_n)$ is
called a {\it path}, if the vertices $\xi_i$ and $\xi_{i+1}$ are
adjacent for any $i=1, \, \dots , \, n-1$. If all the vertices
$\xi_i$ are distinct, then such a path is called {\it simple}. A
path $(\xi_1, \, \dots, \, \xi_n)$ is called a {\it cycle}, if
$n\ge 4$, the path $(\xi_1, \, \dots, \, \xi_{n-1})$ is simple and
$\xi_1=\xi_n$. We say that a graph is {\it connected} if any two
vertices are connected by a finite path. A connected graph is a
{\it tree} if it has no cycles.

Let $({\cal T}, \, \xi_0)$  be a tree with a distinguished vertex
(or a root) $\xi_0$. We introduce a partial order on ${\bf
V}({\cal T})$ as follows: we say that $\xi'>\xi$ if there exists a
simple path $(\xi_0, \, \xi_1, \, \dots , \, \xi_n, \, \xi')$ such
that $\xi=\xi_k$ for some $k\in \overline{0, \, n}$. In this case,
we set $\rho_{{\cal T}}(\xi, \, \xi')=\rho_{{\cal T}}(\xi', \,
\xi) =n+1-k$. In addition, we denote $\rho_{{\cal T}}(\xi, \,
\xi)=0$. If $\xi'>\xi$ or $\xi'=\xi$, then we write $\xi'\ge \xi$.
This partial order on ${\cal T}$ induces a partial order on its
subtree.

For $j\in \Z_+$, $\xi\in {\bf V}({\cal T})$ we denote
$$
\label{v1v}{\bf V}_j(\xi):={\bf V}_j ^{{\cal T}}(\xi):=
\{\xi'\ge\xi:\; \rho_{{\cal T}}(\xi, \, \xi')=j\}.
$$
Given $\xi\in {\bf V}({\cal T})$, we denote by ${\cal
T}_\xi=({\cal T}_\xi, \, \xi)$ a subtree of ${\cal T}$ with the
set of vertices
\begin{align}
\label{vpvtvpv} \{\xi'\in {\bf V}({\cal T}):\xi'\ge \xi\}.
\end{align}

Let ${\cal G}$ be a subgraph in ${\cal T}$. Denote by ${\bf
V}_{\max} ({\cal G})$ and ${\bf V}_{\min}({\cal G})$ the sets of
maximal and minimal vertices in ${\cal G}$, respectively. Given a
function $f:{\bf V}({\cal G})\rightarrow \R$, we set
\begin{align}
\label{flpg} \|f\|_{l_p({\cal G})}=\left ( \sum \limits _{\xi \in
{\bf V} ({\cal G})}|f(\xi)|^p\right )^{1/p}.
\end{align}
Denote by $l_p({\cal G})$ the space of functions $f:{\bf V}({\cal
G})\rightarrow \R$ with finite norm $\|f\|_{l_p({\cal G})}$.

Let ${\bf W}\subset {\bf V}({\cal T})$. We say that ${\cal
G}\subset {\cal T}$ is a maximal subgraph on the set of vertices
${\bf W}$ if ${\bf V}({\cal G})={\bf W}$ and any two vertices
$\xi'$, $\xi''\in {\bf W}$ adjacent in ${\cal T}$ are also
adjacent in ${\cal G}$.

Let $\{{\cal T}_j\}_{j\in \N}$ be a family of subtrees in ${\cal
T}$ such that ${\bf V}({\cal T}_j)\cap {\bf V}({\cal T}_{j'})
=\varnothing$ for $j\ne j'$ and $\cup _{j\in \N} {\bf V}({\cal
T}_j) ={\bf V}({\cal T})$. Then $\{{\cal T}_j\} _{j\in \N}$ is
called a partition of the tree ${\cal T}$. Let $\xi_j$ be the
minimal vertex of ${\cal T}_j$. We say that the tree ${\cal T}_s$
succeeds the tree ${\cal T}_j$ (or ${\cal T}_j$ precedes the tree
${\cal T}_s$) if $\xi_j<\xi_s$ and $$\{\xi\in {\cal T}:\; \xi_j\le
\xi<\xi_s\} \subset {\bf V}({\cal T}_j).$$

\section{Preliminaries}
Let $\Delta$ be a cube with the side length $2^{-m}$, $m\in \Z$.
We denote ${\bf m}(\Delta)=m$. In particular, if $\Delta \in
\Xi\left(\left[-\frac 12, \, \frac 12\right]^d\right)$, then $\Delta
\in \Xi _{{\bf m}(\Delta)}\left(\left[-\frac 12, \, \frac
12\right]^d\right)$.

We shall need Whitney's covering theorem (see, e.g., \cite[p.
562]{leoni1}).
\begin{trma}
\label{whitney} Let $\Omega\subset \left[-\frac 12, \, \frac
12\right]^d$ be an open set. Then there exists a family of closed pairwise
non-overlapping cubes $\Theta(\Omega)=
\{\Delta_j\}_{j\in\N}\subset \Xi\left(\left[-\frac 12, \, \frac
12\right]^d\right)$ with the following properties:
\begin{enumerate}
\item $\Omega=\cup _{j\in \N}\Delta_j$;
\item ${\rm dist}\, (\Delta_j, \, \partial \Omega)
\underset{d}{\asymp} 2^{-{\bf m}(\Delta_j)}$;
\item if $\dim (\Delta_i\cap \Delta_j)=d-1$, then
$$
{\bf m}(\Delta_j)-2\le {\bf m}(\Delta_i)\le {\bf m}(\Delta_j)+2.
$$
\end{enumerate}
\end{trma}

Let $p\ge q$, $c_1\ge c_2\ge \dots \ge c_\nu>0$, ${\bf c}= (c_1,
\, \dots, \, c_\nu)$, $c_{\nu+1}=0$,
$$
B^\nu_p({\bf c}) = \left\{(x_1, \, \dots, \, x_\nu)\in \R^\nu:\;
\left(\frac{x_1}{c_1}, \, \dots, \, \frac{x_\nu}{c_\nu}
\right)\in B_p^\nu\right\}.
$$
Pietsch \cite{pietsch1} and Stesin \cite{stesin}
proved that for $0\le n\le \nu$
\begin{align}
\label{diag_pietsh} \vartheta_n(B_p^\nu({\bf c}), \,
l_q^\nu)=\left\{
\begin{array}{l}\left(\sum \limits _{j=n+1} ^\nu
c_j^{\frac{pq}{p-q}}\right)^{\frac 1q-\frac 1p} \quad
\text{for}\quad p>q,
\\ c_{n+1} \quad \text{ for }p=q.\end{array}\right.
\end{align}
In particular,
\begin{align}
\label{width_pietsch_stesin} \vartheta_n(B_p^\nu, \, l_q^\nu)=
(\nu -n)^{\frac 1q-\frac 1p}, \; n\in \Z_+, \; \nu \in \N, \; n\le
\nu.
\end{align}

For $p<q$, we shall use the following order estimates of widths
$\vartheta_n(B_p^\nu, \, l_q^\nu)$ (see \cite{bib_kashin,
bib_gluskin}).

\begin{trma} \label{glus_trm} If $1<p< q<\infty$, then
\begin{align}
\label{gluskin} d_n(B_p^\nu, \, l_q^\nu)\underset{q,p}{\asymp}
\Phi(n, \, \nu, \, p, \, q),
\end{align}
\begin{align}
\label{gluskin_lin} \lambda _n(B_p^\nu, \,
l_q^\nu)\underset{q,p}{\asymp} \Psi(n, \, \nu, \, p, \, q),
\end{align}
\begin{align}
\label{gluskin_gelf} d^n(B_p^\nu, \,
l_q^\nu)\underset{q,p}{\asymp} \Phi(n, \, \nu, \, q', \, p'),
\end{align}
where
$$\Phi(n, \, \nu, \, p, \, q)=\left\{
\begin{array}{l} \min\bigl\{ 1, \, \left(\nu^{1/q}n^{-1/2}\right)
^{\left(\frac1p-\frac1q\right)/\left(\frac12-\frac1q\right)}\bigr\},
\;
2\le p< q< \infty, \\
\max\bigl\{\nu ^{\frac 1q-\frac 1p}, \, \min \bigl(1, \, \nu
^{\frac 1q}n^{-\frac 12}
\bigr)\bigl(1-\frac{n}{\nu}\bigr)^{1/2}\bigr\}, \; 1< p< 2< q< \infty , \\
\max\bigl\{\nu ^{\frac1q-\frac1p}, \,
\bigl(1-\frac{n}{\nu}\bigr)^{\left(\frac1q-
\frac1p\right)/\left(1-\frac{2}{p}\right)}\bigl\}, \; 1< p< q\le
2,\end{array}\right.$$
$$\Psi(n, \, \nu, \, p, \, q)=\left\{ \begin{array}{l}
\Phi(n, \, \nu, \, p, \, q),
\text{ if }q\le p', \\
\Phi(n, \, \nu, \, q', \, p'),\text{ if }p'<q.\end{array}\right.
$$
\end{trma}

The following theorem is proved in \cite{sobol38}; see also
\cite[p. 566]{leoni1}  and \cite[p. 51]{mazya1}.
\begin{trma}
\label{adams_etc} Let $1<p<q<\infty$, $d\in \N$, $r>0$, $\frac
rd+\frac{1}{q}-\frac{1}{p}=0$. Then the operator
$$
Tf(x)=\int \limits _{\R^d} f(y)|x-y|^{r-d}\, dy
$$
is bounded from $L_p(\R^d)$ in $L_q(\R^d)$.
\end{trma}

\begin{trma}
\label{sob_dn} {\rm (see, e.g., \cite{birm, de_vore_sharpley,
bib_kashin, bibl6}).} Let $r\in \N$, $1< p\le \infty$, $1\le q<
\infty$, $\displaystyle \delta:=r +\frac dq-\frac dp>0$. Set
\begin{align}
\label{sob_dn_theta_pqrd}
\theta _{p,q,r,d}=\left\{\begin{array}{l}
\frac{\delta}{d}-\left(\frac 1q-\frac 1p\right)_+,
\quad \quad\quad\mbox{ if }\ p\ge q\quad
\text{or}\quad p<q, \; \hat q\le 2, \\
\min \bigl\{\frac{\delta}{d}+\min\bigl\{\frac 12-\frac{1}{\hat q},
\, \frac 1p-\frac 1q\bigr\}, \, \frac{\hat q\delta}{2d}\bigr\},
\;\; \mbox{ if }\ p<q, \; \hat q>2.
\end{array}\right.
\end{align}
In addition, we suppose that
$\frac{\delta}{d}+\min\bigl\{\frac 12-\frac{1}{\hat
q}, \, \frac 1p-\frac 1q\bigr\}\ne\frac{\hat q\delta}{2d}$ for
$p<q$, $\hat q>2$. Then
$$
\vartheta_n(W^r_p([0, \, 1]^d), \, L_q([0, \, 1]^d))
\underset{r,d,p,q}{\asymp}n^{-\theta_{p,q,r,d}}.
$$
\end{trma}

The following lemma is proved in \cite{vas_bes}.
\begin{Lem}
\label{sum_lem} Let $\Phi:(0, \, +\infty) \rightarrow (0, \,
+\infty)$, $\rho:(0, \, +\infty) \rightarrow (0, \, +\infty)$ be
absolutely continuous functions and $\lim \limits_{t\to
+0}\frac{t\Phi'(t)}{\Phi(t)}=0$, $\lim \limits_{y\to
+\infty}\frac{y\rho'(y)}{\rho(y)}=0$. Then for any $\varepsilon
>0$
$$
t^{\varepsilon}\underset{\varepsilon,\Phi}{\lesssim}
\frac{\Phi(ty)}{\Phi(y)}\underset{\varepsilon,\Phi}{\lesssim}
t^{-\varepsilon}, \quad 0<y\le 1, \;\; 0<t\le 1,
$$
$$
t^{-\varepsilon} \underset{\varepsilon,\rho}{\lesssim}
\frac{\rho(ty)}{\rho(y)}\underset{\varepsilon,\rho}{\lesssim}
t^\varepsilon,\quad 1\le y<\infty, \;\; 1\le t<\infty.
$$
\end{Lem}

\section{Widths of function classes on a set with \\
a tree-like structure: the general theorem}

Let $(\Omega, \, \Sigma, \, {\rm mes})$ be a measure space, let
$\hat\Theta$ be at most countable partition of $\Omega$ into
measurable subsets, let ${\cal A}$ be a tree such that
\begin{align}
\label{c_v1_a} \exists c_1\ge 1:\quad {\rm card}\, {\bf
V}_1(\xi)\le c_1, \quad \xi \in {\bf V}({\cal A}),
\end{align}
and let $\hat F:{\bf V}({\cal A}) \rightarrow \hat\Theta$ be a
bijective mapping.

Throughout we consider at most countable partitions.

Let $1\le p\le \infty$, $1\le q\le \infty$. We suppose that, for
any measurable subset $E\subset \Omega$, the following spaces are
defined:
\begin{itemize}
\item the space $X_p(E)$ with seminorm $\|\cdot\|_{X_p(E)}$,
\item the space $Y_q(E)$ with norm $\|\cdot\|_{Y_q(E)}$,
\end{itemize}
which all satisfy the following conditions:
\begin{enumerate}
\item $X_p(\Omega)\subset Y_q(\Omega)$;
\item $X_p(E)=\{f|_E:\; f\in X_p(\Omega)\}$, $Y_q(E)=\{f|_E:\; f\in
Y_q(\Omega)\}$;
\item if ${\rm mes}\, E=0$, then $\dim \, Y_q(E)=0$;
\item if $E\subset \Omega$, $E_j\subset \Omega$ ($j\in \N$)
are measurable sets, $E=\sqcup _{j\in \N} E_j$, then
\begin{align}
\label{f_xp} \|f\|_{X_p(E)}=\left\| \bigl\{
\|f|_{E_j}\|_{X_p(E_j)}\bigr\}_{j\in \N}\right\|_{l_p},\quad f\in
X_p(E),
\end{align}
\begin{align}
\label{f_yq} \|f\|_{Y_q(E)}=\left\| \bigl\{\|f|_{E_j}\|
_{Y_q(E_j)}\bigr\}_{j\in \N}\right\|_{l_q}, \quad f\in Y_q(E).
\end{align}
\end{enumerate}
Denote by $BX_p(\Omega)$ the unit ball of the space $X_p(\Omega)$.

Let ${\cal P}(\Omega)\subset X_p(\Omega)$ be a subspace
of the finite dimension $r_0$. For each measurable subset
$E\subset \Omega$ we write ${\cal P}(E)=\{P|_E:\; P\in {\cal
P}(\Omega)\}$. Let $G\subset \Omega$ be a measurable subset
and let $T$ be a partition of $G$. We set
\begin{align}
\label{st_omega} {\cal S}_{T}(\Omega)=\{f:\Omega\rightarrow \R:\,
f|_E\in {\cal P}(E), \; f|_{\Omega\backslash G}=0\}.
\end{align}
For any finite partition $T=\{E_j\}_{j=1}^n$ of the set $E$ and
for each function $f\in Y_q(\Omega)$ we put
$$
\|f\|_{p,q,T}=\left(\sum \limits _{j=1}^n \|f|_{E_j}\|_{Y_q(E_j)}
^{\sigma_{p,q}}\right)^{\frac{1}{\sigma_{p,q}}},
$$
where $\sigma_{p,q}=\min\{p, \, q\}$. Denote by $Y_{p,q,T}(E)$
the space $Y_q(E)$ with the norm $\|\cdot\|_{p,q,T}$. Notice that
$\|\cdot\| _{Y_q(E)}\le \|\cdot\|_{p,q,T}$.

For any subtree ${\cal A}'\subset {\cal A}$ we set
$\Omega _{{\cal A}'}=\cup _{\xi\in {\bf V}({\cal A}')} \hat
F(\xi)$.

\begin{Sup}
\label{sup1} There is a function $w_*:{\bf V}({\cal A})\rightarrow
(0, \, \infty)$ with the following property: for any subtree
${\cal A}'\subset {\cal A}$ rooted at $\hat\xi$ there exists a
linear continuous operator $P_{\Omega_{{\cal
A}'}}:Y_q(\Omega_{{\cal A}'})\rightarrow {\cal P}(\Omega_{{\cal
A}'})$ such that for any function $f\in X_p(\Omega_{{\cal A}'})$
\begin{align}
\label{f_pom_f} \|f-P_{\Omega_{{\cal A}'}}f\|_{Y_q(\Omega_{{\cal
A}'})}\le w_*(\hat\xi)\|f\| _{X_p(\Omega_{{\cal A}'})}.
\end{align}
\end{Sup}

\begin{Sup}
\label{sup2} There exist a function $\tilde w_*:{\bf V}({\cal
A})\rightarrow (0, \, \infty)$ and numbers $\delta_*>0$, $c_2\ge
1$ such that for each vertex $\xi\in {\bf V}({\cal A})$ and for
any $n\in \N$, $m\in \Z_+$ there is a partition $T_{m,n}(G)$ of
the set $G=\hat F(\xi)$ with the following properties:
\begin{enumerate}
\item ${\rm card}\, T_{m,n}(G)\le c_2\cdot 2^mn$.
\item For any $E\in T_{m,n}(G)$ there exists a linear continuous operator
$P_E:Y_q(\Omega)\rightarrow {\cal P}(E)$ such that for any
function $f\in X_p(E)$
\begin{align}
\label{fpef} \|f-P_Ef\|_{Y_q(E)}\le (2^mn)^{-\delta_*}\tilde w_*(\xi)
\|f\| _{X_p(E)}.
\end{align}
\item For any $E\in T_{m,n}(G)$
\begin{align}
\label{ceptm} {\rm card}\,\{E'\in T_{m\pm 1,n}(G):\, {\rm
mes}(E\cap E') >0\} \le c_2.
\end{align}
\end{enumerate}
\end{Sup}

\begin{Sup}
\label{sup3} There exist $k_*\in \N$, $\lambda_*\ge 0$,
\begin{align}
\label{mu_ge_lambda}
\mu_*\ge \lambda_*,
\end{align}
$\gamma_*>0$, absolutely continuous functions $u_*:(0, \, \infty)
\rightarrow (0, \, \infty)$ and $\psi_*:(0, \, \infty) \rightarrow
(0, \, \infty)$, $t_0\in \N$, a partition $\{{\cal
A}_{t,i}\}_{t\ge t_0, \, i\in \hat J_t}$ of the tree ${\cal A}$
such that $\lim \limits _{y\to \infty} \frac{yu_*'(y)}{u_*(y)}=0$,
$\lim \limits _{y\to \infty} \frac{y\psi_*'(y)}{\psi_*(y)}=0$,
\begin{align}
\label{w_s_2} c_3^{-1} 2^{-\lambda_*k_*t}u_*(2^{k_*t}) \le
w_*(\xi)\le c_3\cdot 2^{-\lambda_*k_*t}u_*(2^{k_*t}), \quad \xi
\in {\bf V}({\cal A}_{t,i}),
\end{align}
\begin{align}
\label{til_w_s_2} c_3^{-1} 2^{-\mu_*k_*t}u_*(2^{k_*t}) \le \tilde
w_*(\xi)\le c_3\cdot 2^{-\mu_*k_*t}u_*(2^{k_*t}), \quad \xi \in
{\bf V}({\cal A}_{t,i}),
\end{align}
\begin{align}
\label{nu_t_k} \nu_t:= \sum \limits _{i\in \hat J_t} {\rm card}\,
{\bf V}({\cal A}_{t,i})\le c_3\cdot 2^{k_*\gamma_*t}
\psi_*(2^{k_*t})=: c_3 \overline{\nu}_t,\quad t\ge t_0.
\end{align}
In addition, we assume that the following assertions hold.
\begin{enumerate}
\item If $p>q$, then
\begin{align}
\label{bipf4684gn} 2^{-\lambda_*k_*t} ({\rm card}\, \hat
J_t)^{\frac 1q-\frac 1p} \le c_3\cdot
2^{-\mu_*k_*t}\overline{\nu}_t^{\frac 1q-\frac 1p}.
\end{align}
\item Let $t$, $t'\in \Z_+$. Then
\begin{align}
\label{2l} 2^{-\lambda_*k_*t'}u_*(2^{k_*t'})\le c_3\cdot
2^{-\lambda_*k_*t}u_*(2^{k_*t}), \quad\text{if}\quad  t'\ge t,
\end{align}
\begin{align}
\label{2ll}
\begin{array}{c} 2^{-\mu_*k_*t'}u_*(2^{k_*t'})\left(2^{k_*\gamma_*t'}
\psi_*(2^{k_*t'})\right)^{\frac 1q-\frac 1p}\le \\ \le c_3\cdot
2^{-\mu_*k_*t}u_*(2^{k_*t})\left(2^{k_*\gamma_*t}
\psi_*(2^{k_*t})\right)^{\frac 1q-\frac 1p}, \quad\text{if}\quad
t'\ge t, \quad p>q. \end{array}
\end{align}

\item If the tree ${\cal A}_{t',i'}$ succeeds the tree ${\cal
A}_{t,i}$, then $t'=t+1$.
\end{enumerate}
\end{Sup}
\begin{Rem}
\label{rem_suc} If $\xi \in {\bf V}({\cal A}_{t,i})$, $\xi' \in
{\bf V}({\cal A}_{t',i'})$, $\xi'>\xi$, then $t'>t$.
\end{Rem}

\label{ati_label}Let us make some notations.
\begin{itemize}
\item $\hat \xi_{t,i}$ is the minimal vertex of the tree ${\cal
A}_{t,i}$.
\item $\Gamma _t$ is the maximal subgraph in
${\cal A}$ on the set of vertices $\cup _{i\in \hat J_t} {\bf
V}({\cal A}_{t,i})$, $t\ge t_0$; for $1\le t<t_0$ we put
$\Gamma_t=\varnothing$.
\item $G_t=\cup_{\xi\in {\bf V}(\Gamma_t)}\hat F(\xi)=\cup_{i\in \hat
J_t} \Omega_{{\cal A}_{t,i}}$.
\item $\tilde \Gamma_t$ is the maximal subgraph
on the set of vertices $\cup _{j\ge t} {\bf V}(\Gamma_j)$, $t\in
\N$.
\item $\tilde {\cal A}_{t,i}$ ($i\in \overline{J}_t$)
is the set of connected components of the graph $\tilde \Gamma _t$.
\item $\tilde U_{t,i}=\cup _{\xi\in {\bf V}(\tilde{\cal
A}_{t,i})}\hat F(\xi)$.
\item $\tilde U_t =\cup _{i\in \overline{J}_t} \tilde U_{t,i} =\cup
_{\xi\in {\bf V}(\tilde \Gamma_t)} \hat F(\xi)$.
\end{itemize}
We claim that if $t\ge t_0$, then
\begin{align}
\label{v_min_gt} {\bf V}_{\min}(\tilde \Gamma_t)= {\bf
V}_{\min}(\Gamma_t)=\{\hat \xi_{t,i}\}_{i\in \hat J_t}.
\end{align}
Indeed, let $t=t_0$, $i_0\in \hat J_{t_0}$. Then $\hat
\xi_{t_0,i_0}$ is the minimal vertex in the tree ${\cal A}$
(otherwise there is a tree ${\cal A}_{t',i'}$ that precedes ${\cal
A}_{t_0,i_0}$; by Assumption \ref{sup3}, $t'=t_0-1$, which leads
to a contradiction). Thus, $\hat J_{t_0}=\{i_0\}$ and ${\bf
V}_{\min}(\tilde \Gamma_{t_0})={\bf V}_{\min}({\cal A})=\{\hat
\xi_{t_0,i_0}\}={\bf V}_{\min}( \Gamma_{t_0})$.

Let $t>t_0$. Notice that ${\bf V}_{\min}(\Gamma_t)\subset \{\hat
\xi_{t,i}\}_{i\in \hat J_t}$. We claim that ${\bf
V}_{\min}(\tilde\Gamma_t)\subset \{\hat \xi_{t,i}\}_{i\in \hat
J_t}$. In fact, let $\xi\in {\bf V}_{\min}(\tilde\Gamma _t)$. Then
$\xi=\hat \xi_{t',i'}$ for some $t'\ge t$, $i'\in \hat J_{t'}$.
Since $t>t_0$, $\hat\xi_{t',i'}> \hat \xi_{t_0,i_0}$. Let the tree
${\cal A}_{t'',i''}$ precede the tree ${\cal A}_{t',i'}$. Then
$t''<t$. By Assumption \ref{sup3}, $t'=t''+1$. Therefore,
$t'<t+1$, i.e., $t'=t$. Let us prove that ${\bf V}_{\min}
(\Gamma_t) \supset \{\hat \xi_{t,i}\}_{i\in \hat J_t}$ and ${\bf
V}_{\min} (\tilde \Gamma_t) \supset \{\hat \xi_{t,i}\}_{i\in \hat
J_t}$. Indeed, since $\hat\xi_{t,i}$ is the minimal vertex of
${\cal A} _{t,i}$, we have $\hat\xi_{t,i} \in {\bf V}(\tilde\Gamma
_t)$. If $\xi<\hat \xi_{t,i}$, then $\xi\in {\bf V}({\cal
A}_{t',i'})$, $t'<t$ (see Remark \ref{rem_suc}). Hence,
$\hat\xi_{t,i} \in {\bf V}_{\min}(\tilde\Gamma _t)$ and
$\hat\xi_{t,i} \in {\bf V}_{\min}(\Gamma _t)$.

Thus, we may assume that
\begin{align}
\label{ovrl_it_eq_hat_it} \overline{J}_t=\hat J_t, \quad t\ge t_0.
\end{align}

\begin{Lem}
\label{obr} There exists $x_0\in (0, \, \infty)$ such that for any $x\ge
x_0$ the equation $y^{\gamma_*}\psi_*(y)=x$ has the unique solution
$y(x)$. Moreover, $y(x)=x^{\beta_*}\varphi_*(x)$, where
$\beta_*=\frac{1}{\gamma_*}$ and $\varphi_*$ is an absolutely
continuous function such that $\lim _{x\to +\infty}
\frac{x\varphi_*'(x)}{\varphi_*(x)}=0$.
\end{Lem}
This lemma will be proved later.

Denote $\mathfrak{Z}_0=(p, \, q, \, r_0,\, w_*, \, \tilde w_*, \,
\delta_*, \, k_*, \, \lambda_*, \, \mu_*,\, \gamma_*, \, \psi_*,\,
u_*,\, c_1, \, c_2, \, c_3)$.
\begin{Trm}
\label{main_abstr_th} Suppose that Assumptions \ref{sup1},
\ref{sup2} and \ref{sup3} hold and $\delta_*>\left(\frac 1q-\frac
1p\right)_+$. Then there exists $n_0=n_0(\mathfrak{Z}_0)$ such
that for any $n\ge n_0$ the following estimates are true.
\begin{enumerate}
\item Let $\delta_*\ne \lambda_*\beta_*$.
\begin{itemize}
\item For $p\le q$, we set
$$
\sigma_*(n)=\left\{ \begin{array}{l} 1, \quad \text{if}\;
\delta_*<\lambda_*\beta_*, \\ u_*(n^{\beta_*} \varphi_*(n))
\varphi_*^{-\lambda_*}(n), \quad \text{if}\; \delta_* >
\lambda_*\beta_*.\end{array}\right.
$$
Then
\begin{align}
\label{vrth_n_pleq} \vartheta_n(BX_p(\Omega), \, Y_q(\Omega))
\underset{\mathfrak{Z}_0} {\lesssim} n^{-\min(\delta_*, \,
\lambda_*\beta_*)}\sigma_*(n).
\end{align}
\item For $p>q$, we set
$$
\sigma_*(n)=\left\{ \begin{array}{l} 1, \quad \text{if}\;
\delta_*<\mu_*\beta_*, \\ u_*(n^{\beta_*} \varphi_*(n))
\varphi_*^{-\mu_*}(n), \quad \text{if}\; \delta_* >
\mu_*\beta_*.\end{array}\right.
$$
Then
\begin{align}
\label{vrth_n_pgq} \vartheta_n(BX_p(\Omega), \, Y_q(\Omega))
\underset{\mathfrak{Z}_0} {\lesssim} n^{-\min(\delta_*, \,
\mu_*\beta_*)+\frac 1q-\frac 1p}\sigma_*(n).
\end{align}
\end{itemize}
\item Let $p<q$, $\hat q>2$. Set $\theta_1=\delta_*+\min \left\{ \frac 12-
\frac{1}{\hat q}, \, \frac 1p-\frac 1q\right\}$, $\theta_2
=\frac{\delta _*\hat q}{2}$, $\theta_3 =\lambda_*\beta_* +\min \left\{
\frac 12- \frac{1}{\hat q}, \, \frac 1p-\frac 1q\right\}$,
$\theta_4 =\frac{\lambda_*\beta_* \hat q}{2}$,
$\sigma_1(n)=\sigma_2(n)\equiv 1$,
$\sigma_3(n)=u_*(n^{\beta_*}\varphi_*(n))\varphi_*^{-\lambda_*}(n)$,
$\sigma_4(n)= \sigma_3(n^{\frac{\hat q}{2}})$. Suppose that there exists
$j_*\in \{1, \, 2, \, 3, \, 4\}$ such that $\theta_{j_*}<\min
_{j\ne j_*} \theta_j$. Then
$$
\vartheta_n(BX_p(\Omega), \, Y_q(\Omega))
\underset{\mathfrak{Z}_0} {\lesssim} n^{-\theta_{j_*}}
\sigma_{j_*}(n).
$$
\end{enumerate}
\end{Trm}

\renewcommand{\proofname}{\bf Proof of Lemma \ref{obr}}
\begin{proof}
Since $\gamma_*>0$, by Lemma \ref{sum_lem} we have
$y^{\gamma_*}\psi_*(y)\underset {y\to \infty}{\to} +\infty$. If
$y>0$ is sufficiently large, then
$$
(y^{\gamma_*}\psi_*(y))'=\gamma_* y^{\gamma_*-1}\psi_*(y)+
y^{\gamma_*}\psi_*'(y)>0.
$$
This implies the first part of Lemma.

Prove that the function $\varphi_*(x)$ is absolutely continuous
for large $x$. To this end it is sufficient to prove that the
function $y(x)$ is locally Lipschitz. Choose $z_0>0$ such that
$\left|\frac{t\psi_*'(t)}{\psi_*(t)}\right|<\frac{\gamma_*}{2}$
for any $t\ge z_0$. For each $\overline{z}> z_0$ we take
$\varepsilon_{\overline{z}}\in \left(0, \,
\frac{\overline{z}-z_0}{2} \right)$ such that $\psi_*(z)\ge
\frac{\psi_*(\overline{z})} {2}>0$ for any $z\in
[\overline{z}-\varepsilon_{\overline{z}}, \, \overline{z}+
\varepsilon_{\overline{z}}]$. Estimate the quantity
$(z+h)^{\gamma_*}\psi_*(z+h)-z^{\gamma_*}\psi_*(z)$ from below for
$\overline{z}-\varepsilon_{\overline{z}}\le z\le z+h\le
\overline{z}+ \varepsilon _{\overline{z}}$. We have
$$
(z+h)^{\gamma_*}\psi_*(z+h)-z^{\gamma_*}\psi_*(z) = \int \limits
_z^{z+h} (\gamma_* t^{\gamma_*-1}\psi_*(t)+t^{\gamma_*}
\psi_*'(t))\, dt=
$$
$$
=\int \limits _z^{z+h} t^{\gamma_*-1}\psi_*(t) \left(
\gamma_*+\frac{t\psi_*'(t)}{\psi_*(t)}\right)\, dt \ge
\frac{\gamma_*}{2}\int \limits _z^{z+h} t^{\gamma_*-1}\psi_*(t)\,
dt\underset{\gamma_*,\psi_*,\, \overline{z}}{\gtrsim} h.
$$

Let us prove that $\frac{x\varphi_*'(x)}{\varphi_*(x)}
\underset{x\to +\infty}{\to} 0$. We denote $y=y(x)$. From the
identity $y^{\gamma_*} \psi_*(y)=x$ we get
$$
(\gamma_*y^{\gamma_*-1}\psi_*(y)+y^{\gamma_*}\psi_*'(y))y'(x)=1.
$$
Further, $\varphi_*(x)=x^{-\beta_*}y(x)$, which implies
$$
\varphi_*'(x)=-\beta_*x^{-\beta_*-1}y(x)+x^{-\beta_*}y'(x)=
-\beta_*\frac{\varphi_*(x)}{x}+x^{-\beta_*}y'(x)=
$$
$$
=-\beta_*\frac{\varphi_*(x)}{x}+\frac{x^{-\beta_*}}{\gamma_*y^{\gamma_*-1}
\psi_*(y)+y^{\gamma_*}\psi_*'(y)}.
$$
Hence,
$$
\frac{x\varphi'_*(x)}{\varphi_*(x)}=-\beta_*+
\frac{x^{1-\beta_*}}{\varphi_*(x)(\gamma_*y^{\gamma_*-1}
\psi_*(y)+y^{\gamma_*}\psi_*'(y))}=
$$
$$
=-\beta_*+\frac{x}{y(\gamma_*y^{\gamma_*-1}
\psi_*(y)+y^{\gamma_*}\psi_*'(y))}=-\beta_*+\frac{y^{\gamma_*-1}\psi_*(y)}
{\gamma_*y^{\gamma_*-1} \psi_*(y)+y^{\gamma_*}\psi_*'(y)}=
$$
$$
=-\frac{1}{\gamma_*}+\frac{1}{\gamma_*+\frac{y(x)\psi'_*(y(x))}{\psi_*(y(x))}}
\underset{x\to+\infty}{\to} 0.
$$
This completes the proof.
\end{proof}
\renewcommand{\proofname}{\bf Proof}

\begin{Lem}
\label{log} Let $\gamma_*>0$, $\psi_*(y)=|\log
y|^{\alpha_*}\rho_*(|\log y|)$, where $\rho_*:(0, \, \infty)
\rightarrow (0, \, \infty)$ is an absolutely continuous function
such that $\lim \limits _{y\to \infty} \frac{y
\rho_*'(y)}{\rho_*(y)}=0$. Let $\varphi_*$ be such as in Lemma
\ref{obr}. Then for sufficiently large $x>1$
$$
\varphi_*(x)\underset{\gamma_*, \alpha_*,\rho_*}{\asymp} (\log
x)^{-\frac{\alpha_*}{\gamma_*}}\left[\rho_*(\log
x)\right]^{-\frac{1}{\gamma_*}}.
$$
\end{Lem}
\begin{proof}
Set $\hat y(x)=x^{\frac{1}{\gamma_*}} (\log x)
^{-\frac{\alpha_*}{\gamma_*}} \rho_*^{-\frac{1}{\gamma _*}} (\log
x)$. For sufficiently large $x>1$ we have $\log \hat
y(x)\underset{\gamma_*,\alpha_*,\rho_*}{\asymp } \log x$. Hence,
$$
(\hat y(x))^{\gamma_*} (\log \hat y(x))^{\alpha_*} \rho_*(\log
\hat y(x)) \underset{\gamma_*,\alpha_*,\rho_*}{\asymp} x.
$$
Applying Lemma \ref{sum_lem}, we get for $c>1$
$$
(c\hat y(x))^{\gamma_*}\psi_*(c\hat y(x))
\underset{\gamma_*,\alpha_*,\rho_*}{\gtrsim} c^{\gamma_*/2} (\hat
y(x))^{\gamma_*} \psi_*(\hat y(x))
\underset{\gamma_*,\alpha_*,\rho_*}{\asymp} c^{\gamma_*/2} x,
$$
$$
(c^{-1}\hat y(x))^{\gamma_*}\psi_*(c^{-1}\hat y(x))
\underset{\gamma_*,\alpha_*,\rho_*}{\lesssim} c^{-\gamma_*/2}
(\hat y(x))^{\gamma_*} \psi_*(\hat y(x))
\underset{\gamma_*,\alpha_*,\rho_*}{\asymp} c^{-\gamma_*/2} x,
\;\; \text{if}\;\;c^{-1}\hat y(x)>1.
$$
Consequently, there exists $c=c(\gamma_*,\alpha_*,\rho_*)>1$ such
that for sufficiently large $x>1$
$$
(c\hat y(x))^{\gamma_*}\psi_*(c\hat y(x))>x, \quad (c^{-1}\hat
y(x))^{\gamma_*}\psi_*(c^{-1}\hat y(x))<x.
$$
Thus, if $y^{\gamma_*}\psi_*(y)=x$, then by monotonicity
we obtain $y\in [c^{-1}\hat y(x), \, c\hat y(x)]$.
\end{proof}

The following Lemma was proved in \cite{vas_john} (here we use the
special case of finite trees).
\begin{Lem}
\label{lemma_o_razb_dereva1} Let $({\cal T}, \, \xi_*)$ be a tree
with finite number of vertices,
\begin{align}
\label{cardvvvk} {\rm card}\, {\bf V}_1(\xi)\le k, \;\; \xi\in
{\bf V}({\cal T}),
\end{align}
and let $\Phi :2^{{\bf V}(\cal T)}\rightarrow \R_+$ satisfy
the following conditions:
\begin{align}
\label{prop_psi} \Phi(V_1\cup V_2)\ge \Phi(V_1)+\Phi(V_2), \; V_1,
\, V_2\subset {\bf V}({\cal T}), \;\; V_1\cap V_2=\varnothing,
\end{align}
$\Phi({\bf V}({\cal T}))>0$. Then there exists a constant $C(k)>0$
such that for any $n\in \N$ there is a partition $\mathfrak{S}_n$
of the tree ${\cal T}$ into at most $C(k)n$ subtrees ${\cal T}_j$
satisfying the following properties:
\begin{enumerate}
\item $\Phi({\bf V}({\cal T}_j))\le \frac{(k+2)\Phi({\bf V}({\cal T}))}{n}$ for any
$j$ such that ${\rm card}\, {\bf V}({\cal T}_j)\ge 2$;
\item if $m\le 2n$, then each element of
$\mathfrak{S}_n$ overlaps at most $C(k)$ elements of
$\mathfrak{S}_m$.
\end{enumerate}
\end{Lem}
\begin{Lem}
\label{oper_a} Let $T$ be a finite partition of the measurable
subset $G\subset \Omega$, and let $\nu=\dim {\cal
S}_T(\Omega)$ (see (\ref{st_omega})). Then there exists a linear
isomorphism $A:{\cal S}_T(\Omega)\rightarrow \R^\nu$ such that
$\|A\|_{Y_{p,q,T}(G)\rightarrow l_p^\nu}\underset{p, \,
r_0}{\lesssim} 1$, $\|A^{-1}\| _{l_q^\nu\rightarrow Y_q(G)}
\underset{q, \, r_0}{\lesssim} 1$.
\end{Lem}
Lemma \ref{oper_a} follows from John's ellipsoid theorem (see,
e.g., \cite{pisier}, Chapters 1 and 3). The proof is the same as in
\cite{avas2} (pages 499--501).

\renewcommand{\proofname}{\bf Proof of Theorem \ref{main_abstr_th}}

\begin{proof}
We define $n_0=n_0(\mathfrak{Z})$ so that the assertion of Lemma
\ref{obr} holds for $x\ge n_0$.

{\bf Step 1.} Let $n=2^N$, $N\in \N$, $n\ge n_0$. We set
$t_*(n)=\min\{t\in \N:\; \overline{\nu}_t\ge n\}$. Then
\begin{align}
\label{2_ks_t} \overline{\nu}_{t_*(n)}
\underset{\mathfrak{Z}_0}{\asymp} n, \quad 2^{k_*t_*(n)}
\underset{\mathfrak{Z}_0}{\asymp} n^{\beta_*}\varphi_*(n).
\end{align}
Indeed, $\overline{\nu}_{t_*(n)}\ge n$ by definition. On the other hand,
for any $t\in \N$ we have
$\frac{\overline{\nu}_t}{\overline{\nu}_{t-1}}=2^{k_*\gamma_*}
\frac{\psi_*(2^{k_*t})}{\psi_*(2^{k_*(t-1)})} \underset
{\mathfrak{Z}_0}{\lesssim} 1$ by Lemma \ref{sum_lem}. Since
$\overline{\nu}_{t_*(n)-1}<n$, this implies that
$\overline{\nu}_{t_*(n)} \underset{\mathfrak{Z}_0}{\lesssim} n$.
Therefore, $2^{k_*t_*(n)}=(cn)^{\beta_*}\varphi_*(cn)$,
$c\underset{\mathfrak{Z}_0}{\asymp} 1$. Apply Lemma
\ref{sum_lem} once again and obtain the second relation in (\ref{2_ks_t}).

Let $t_0\le t\le t_*(n)$. Construct the partition $T^1_{m,n,t}$ of the set
$G_t$. We put $\hat t(n)=t_0$ or $\hat t(n)=t_*(n)$ for each $N\in \N$
and take some $\varepsilon>0$ (the choice of $t_*(n)$ and $\varepsilon$
will be made later). Let
\begin{align}
\label{n_t} n_t:=\left\lceil 2^{N-\gamma_*k_*t-\varepsilon|t-\hat
t(n)|} (\psi_*(2^{k_*t}))^{-1}\right \rceil\ge
2^{N-\gamma_*k_*t-\varepsilon|t-\hat
t(n)|} (\psi_*(2^{k_*t}))^{-1},
\end{align}
\begin{align}
\label{m_ts} m_t^*=\min \left\{m\in
\Z_+:2^{N-\gamma_*k_*t-\varepsilon|t-\hat t(n)|+m}
(\psi_*(2^{k_*t}))^{-1}\ge 1\right\}.
\end{align}
For $\xi\in {\bf V}(\Gamma_t)$, $m\ge m_t^*$, we set
$T^1_{m,n,t}|_{\hat F(\xi)}= T_{m-m_t^*,n_t}(\hat F(\xi))$ (see
Assumption \ref{sup2}). If $m_t^*=0$, then $n_t \asymp
2^{N-\gamma_*k_*t -\varepsilon |t-\hat
t(n)|}(\psi_*(2^{k_*t}))^{-1}$. If $m_t^*>0$, then
\begin{align}
\label{nt_eq_12mt_eq_2ngkt}
n_t=1, \quad 2^{-m_t^*}\asymp 2^{N-\gamma_*k_*t
-\varepsilon |t-\hat t(n)|}(\psi_*(2^{k_*t}))^{-1}.
\end{align}
Hence,
\begin{align}
\label{2mmtnut} 2^{-m_t^*}n_t\asymp 2^{N-\gamma_*k_*t -\varepsilon
|t-\hat t(n)|}(\psi_*(2^{k_*t}))^{-1}.
\end{align}
By assertion 1 of Assumption \ref{sup2} and (\ref{nu_t_k}),
\begin{align}
\label{ct1mnt} \begin{array}{c} {\rm card}\, T^1_{m,n,t}
\underset{\mathfrak{Z}_0}{\lesssim} \nu_t\cdot 2^{m-m_t^*}n_t
\stackrel{(\ref{nu_t_k}),(\ref{2mmtnut})}{\underset
{\mathfrak{Z}_0}{\lesssim}} \\
\lesssim 2^{\gamma_*k_*t}\psi_*(2^{k_*t}) 2^m\cdot
2^{N-\gamma_*k_*t -\varepsilon |t-\hat
t(n)|}(\psi_*(2^{k_*t}))^{-1}= 2^mn\cdot 2^{-\varepsilon|t-\hat
t(n)|}. \end{array}
\end{align}
From assertion 2 of Assumption \ref{sup2}, (\ref{til_w_s_2})
and the definition of $\Gamma_t$ it follows that for any
$E\in T^1_{m,n,t}$ there exists a linear continuous
operator $P_E:Y_q(E) \rightarrow {\cal P}(E)$ such that
for any function $f\in X_p(\Omega)$
\begin{align}
\label{fpef_1} \|f-P_Ef\|_{Y_q(E)}
\underset{\mathfrak{Z}_0}{\lesssim}
2^{-\mu_*k_*t}u_*(2^{k_*t})(2^{m-m_t^*}n_t)^{-\delta_*}
\|f\|_{X_p(E)}.
\end{align}
Set $P^1_{m,n,t}f|_E=P_Ef$, $E\in T^1_{m,n,t}$,
$P^1_{m,n,t}f|_{\Omega\backslash G_t}=0$. Then
\begin{align}
\label{p1mnt} P^1_{m,n,t}: Y_q(\Omega)\rightarrow
S_{T^1_{m,n,t}}(\Omega)
\end{align}
is a linear continuous operator, and for any function $f\in
X_p(\Omega)$
\begin{align}
\label{fp1ft}
\|f-P^1_{m,n,t}f\|_{p,q,T^1_{m,n,t}}\stackrel{(\ref{fpef_1})}
{\underset{\mathfrak{Z}_0}{\lesssim}} 2^{-\mu_*k_*t}u_*(2^{k_*t})
(2^{m-m_t^*}n_t)^{-\delta_*} \|f\|_{X_p(G_t)}, \quad \text{if
}p\le q.
\end{align}

By assertion 3 of Assumption \ref{sup2}, for any $E\in
T^1_{m,n,t}$
\begin{align}
\label{card_e} {\rm card}\, \{E'\in T^1_{m\pm 1,n,t}:\, {\rm
mes}(E\cap E')>0\} \underset{\mathfrak{Z}_0}{\lesssim} 1.
\end{align}

{\bf Step 2.} By Assumption \ref{sup1} and (\ref{w_s_2}), for any
$t\ge t_0$, $i\in \overline{J}_t$ there exists a linear continuous
operator $\tilde P_{t,i}: Y_q(\tilde U_{t,i})\rightarrow {\cal
P}(\tilde U_{t,i})$ such that for any function $f\in X_p(\Omega)$
\begin{align}
\label{ftp_ti} \|f-\tilde P_{t,i}f\|_{Y_q(\tilde U_{t,i})}
\underset{\mathfrak{Z}_0}{\lesssim}
2^{-\lambda_*k_*t}u_*(2^{k_*t})\|f\| _{X_p(\tilde U_{t,i})}.
\end{align}
Let $1\le t\le t_0$. Then $\tilde U_{t,i}=\tilde U_{t_0,i_0}$. We
set $\tilde P_{t,i}=\tilde P_{t_0,i_0}$. By (\ref{2l}), we obtain
that (\ref{ftp_ti}) holds as well.

{\bf Step 3.} Let us prove the estimates (\ref{vrth_n_pleq}) and
(\ref{vrth_n_pgq}).

We set $T^1_n|_{G_t}=T^1_{m_t^*,n,t}$, $t_0\le t\le t_*(n)$,
$T^1_n|_{\tilde U_{t_*(n)+1,i}}=\{\tilde U_{t_*(n)+1,i}\}$, $i\in
\overline{J}_{t_*(n)+1}$, $P^1_nf|_{G_t}=P^1_{m_t^*,n,t}f$,
$t_0\le t\le t_*(n)$, $P^1_nf|_{\tilde U_{t_*(n)+1,i}}=\tilde
P_{t_*(n)+1,i}$, $i\in \overline{J}_{t_*(n)+1}$. Then
$P^1_n:Y_q(\Omega)\rightarrow {\cal S}_{T^1_n}(\Omega)$ is a
linear continuous operator. If $t_*(n)<t_0$, then
$T^1_n=\{\Omega\}$; if $t_*(n)\ge t_0$, then
$\overline{J}_{t_*(n)+1}=\hat J_{t_*(n)+1}$ and
$$
{\rm card}\, T^1_n=\sum \limits _{t=t_0}^{t_*(n)} {\rm card}\,
T^1_{m^*_t,n,t}+{\rm card}\, \hat
J_{t_*(n)+1}\stackrel{(\ref{nu_t_k}),(\ref{nt_eq_12mt_eq_2ngkt}),(\ref{ct1mnt})
}{\underset{\mathfrak{Z}_0}{\lesssim}} \sum \limits _{t_0\le t\le
t_*(n), \, m^*_t=0} n\cdot 2^{-\varepsilon|t-\hat t(n)|}+
$$
$$
+\sum \limits _{t_0\le t\le t_*(n),m^*_t>0} \nu_t+\nu_{t_*(n)+1}
\stackrel{(\ref{nu_t_k})}{\underset{\varepsilon,
\mathfrak{Z}_0}{\lesssim}} n+\overline{\nu}_{t_*(n)+1} \stackrel
{(\ref{2_ks_t})}{\underset {\mathfrak{Z}_0}{\lesssim}} n.
$$
Let $p\le q$. Then for $f\in BX_p(\Omega)$, $t_*(n)\ge t_0$ we get
$$
\|f-P^1_nf\|_{Y_q(\Omega)}\le
$$
$$
\le\left(\sum \limits _{t=t_0}^{t_*(n)} \sum \limits _{E\in
T^1_{m^*_t,n,t}} \|f-P_Ef\| _{Y_q(E)}^p +\sum \limits _{i\in \hat
J_{t_*(n)+1}} \|f-\tilde P_{t_*(n)+1,i}f\|^p_{Y_q(\tilde
U_{t_*(n)+1,i})} \right)^{1/p}\stackrel{(\ref{mu_ge_lambda}),
(\ref{fpef_1}),(\ref{ftp_ti})}{\underset{\mathfrak{Z}_0}{\lesssim}}
$$
$$
\lesssim \left(\sum \limits _{t=t_0}^{t_*(n)}\sum \limits _{E\in
T^1_{m^*_t,n,t}} 2^{-\lambda_*k_*tp}u_*^p
(2^{k_*t})n_t^{-\delta_*p} \|f\|^p _{X_p(E)}\right)^{1/p} +
$$
$$
+\left(\sum \limits _{i\in \hat J_{t_*(n)+1}}
2^{-\lambda_*pk_*t_*(n)}u_*^p(2^{k_*t_*(n)})\|f\|^p_{X_p(\tilde
U_{t_*(n)+1,i})}\right)^{1/p}\stackrel{(\ref{2_ks_t}),(\ref{n_t})}
{\underset{\mathfrak{Z}_0}{\lesssim}}
$$
$$
\lesssim \left(\sum \limits _{t=t_0}^{t_*(n)} \sum \limits _{E\in
T^1_{m^*_t,n,t}}
2^{-\lambda_*k_*tp}u_*^p(2^{k_*t})\bigl(2^{N-\gamma_*k_*t
-\varepsilon|t-\hat
t(n)|}(\psi_*(2^{k_*t}))^{-1}\bigr)^{-\delta_*p}
 \|f\| ^p_{X_p(E)}\right)^{1/p}+$$$$
+n^{-\lambda_*\beta_*}u_*(n^{\beta_*}\varphi_*(n))\varphi_*^{-\lambda_*}(n)\le
$$
$$
\le \left(\sum \limits _{t=t_0}^{t_*(n)} 2^{-\lambda_*k_*tp}
u_*^p(2^{k_*t})\bigl(2^{N-\gamma_*k_*t-\varepsilon|t-\hat
t(n)|}(\psi_*(2^{k_*t}))^{-1}\bigr)^{-\delta_*p}\right)^{1/p}+
$$
$$
+n^{-\lambda_*\beta_*}u_*(n^{\beta_*}\varphi_*(n))\varphi_*^{-\lambda_*}(n).
$$

Let $\varepsilon=\frac{|-\lambda_*+\delta_*\gamma_*|}{2\delta_*}$.
If $\delta_*<\beta_*\lambda_*$, then $-\lambda_*+ \delta_*\gamma_*
<0$; in this case we set $\hat t(n)=1$ and obtain the estimate
$$
\left(\sum \limits _{t=t_0}^{t_*(n)}
2^{-\lambda_*k_*tp}u_*^p(2^{k_*t})\bigl(2^{N-\gamma_*k_*t
-\varepsilon|t-\hat
t(n)|}(\psi_*(2^{k_*t}))^{-1}\bigr)^{-\delta_*p}
\right)^{1/p}\underset{\mathfrak{Z}_0}{\lesssim}  n^{-\delta_*}.
$$
If $\delta_*>\beta_*
\lambda_*$, then we take $\hat t(n)=t_*(n)$ and get
$$
\left(\sum \limits _{t=t_0}^{t_*(n)}
2^{-\lambda_*k_*tp}u_*^p(2^{k_*t})\bigl(2^{N-\gamma_*k_*t
-\varepsilon|t-\hat
t(n)|}(\psi_*(2^{k_*t}))^{-1}\bigr)^{-\delta_*p}
\right)^{1/p}\stackrel{(\ref{nu_t_k}),(\ref{2_ks_t})}{\underset{\mathfrak{Z}_0}{\lesssim}}
$$
$$
\lesssim 2^{-\lambda_*k_*t_*(n)}u_*(2^{k_*t_*(n)})
\stackrel{(\ref{2_ks_t})}{\underset{\mathfrak{Z}_0}
{\asymp}}n^{-\lambda_*\beta_*}u_*(n^{\beta_*}
\varphi_*(n))\varphi_*^{-\lambda_*}(n).
$$

If $1\le t_*(n)<t_0$, then for $f\in BX_p(\Omega)$ we obtain
$$
\|f-P^1_nf\|_{Y_q(\Omega)}= \|f-\tilde
P_{t_0,i_0}f\|_{Y_q(\Omega)} \stackrel{(\ref{ftp_ti})}
{\underset{\mathfrak{Z}_0}{\lesssim}}
2^{-\lambda_*k_*t_0}u_*(2^{k_*t_0})
\stackrel{(\ref{2l})}{\underset{\mathfrak{Z}_0}{\lesssim}}
$$
$$
\lesssim 2^{-\lambda_*k_*t_*(n)}u_*(2^{k_*t_*(n)})
\stackrel{(\ref{2_ks_t})}{\underset{\mathfrak{Z}_0}{\asymp}}
n^{-\lambda_*\beta_*}u_*(n^{\beta_*}
\varphi_*(n))\varphi_*^{-\lambda_*}(n).
$$

Let $p>q$. By the conditions of the theorem, $\delta_*>\frac 1q-\frac
1p$. Hence,
\begin{align}
\label{nt1pq} n_t^{1-\delta_*\frac{pq}{p-q}}
\stackrel{(\ref{n_t})}{\le}
\left(2^{N-\gamma_*k_*t-\varepsilon|t-\hat
t(n)|}(\psi_*(2^{k_*t}))^{-1}\right)^{1-\delta_*\frac{pq}{p-q}}.
\end{align}
Applying H\"{o}lder's inequality, we get for $f\in BX_p(\Omega)$,
$t_*(n)\ge t_0$
$$
\|f-P^1_nf\|_{Y_q(\Omega)}=
$$
$$
= \left(\sum \limits _{t=t_0}^{t_*(n)} \sum \limits _{E\in
T^1_{m^*_t,n,t}} \|f-P_Ef\| _{Y_q(E)}^q +\sum \limits _{i\in \hat
J_{t_*(n)+1}} \|f-\tilde P_{t_*(n)+1,i}f\|^q_{Y_q(\tilde
U_{t_*(n)+1,i})} \right)^{1/q}\stackrel{(\ref{fpef_1}),
(\ref{ftp_ti})}{\underset{\mathfrak{Z}_0}{\lesssim}}
$$
$$
\lesssim \left(\sum \limits _{t=t_0}^{t_*(n)}\sum \limits _{E\in
T^1_{m^*_t,n,t}} 2^{-\mu_*k_*tq} u_*^q(2^{k_*t})n_t^{-\delta_*q}
\|f\|^q _{X_p(E)}\right)^{1/q} +
$$
$$
+\left(\sum \limits _{i\in \hat J_{t_*(n)+1}}
2^{-q\lambda_*k_*t_*(n)} u_*^q(2^{k_*t_*(n)})\|f\|^q_{X_p(\tilde
U_{t_*(n)+1,i})}\right)^{1/q} \le
$$
$$
\le \left(\sum \limits _{t=t_0}^{t_*(n)} {\rm card}\,
T^1_{m^*_t,n,t}
2^{-\mu_*k_*t\frac{pq}{p-q}}u_*^{\frac{pq}{p-q}}(2^{k_*t})
n_t^{-\frac{\delta_* pq}{p-q}} \right)^{\frac 1q-\frac 1p} +
$$
$$
+\left({\rm card}\, \hat J_{t_*(n)+1} \cdot 2^{-\frac{pq
}{p-q}\lambda_*k_*t_*(n)}u_*^{\frac{pq}{p-q}}(2^{k_*t_*(n)})
\right)^{\frac 1q-\frac 1p}\stackrel{(\ref{nu_t_k}),
(\ref{ct1mnt}), (\ref{nt1pq})}{\underset
{\mathfrak{Z}_0}{\lesssim}}
$$
$$
\lesssim \left(\sum \limits _{t=t_0}^{t_*(n)} 2^{\gamma_*k_*t}
\psi_*(2^{k_*t}) [2^{-\mu_*k_*t} u_*(2^{k_*t})]^{\frac{pq}{p-q}}
\bigl[ 2^{N-\gamma_*k_*t-\varepsilon|t-\hat
t(n)|}\psi_*^{-1}(2^{k_*t}) \bigr]^{1-\frac{pq}{p-q}\delta_*}
\right)^{\frac 1q-\frac 1p}+
$$
$$
+2^{-\lambda_*k_*t_*(n)} u_*(2^{k_*t_*(n)}) ({\rm card}\, \hat
J_{t_*(n)+1})^{\frac 1q-\frac 1p}\stackrel{(\ref{bipf4684gn}),
(\ref{2_ks_t})}{\underset{\mathfrak{Z}_0}{\lesssim}}
$$
$$
\lesssim \left(\sum \limits _{t=t_0}^{t_*(n)}
\psi_*^{\frac{pq}{p-q}\delta_*}(2^{k_*t}) \cdot
2^{-(\mu_*-\gamma_*\delta_*)k_*t\frac{pq}{p-q}} \bigl[
2^{N-\varepsilon|t-\hat t(n)|} \bigr]^{1-\frac{pq}{p-q}\delta_*}
u_*^{\frac{pq}{p-q}}(2^{k_*t})\right)^{\frac 1q-\frac 1p}+
$$
$$
+n^{-\mu_*\beta_*+\frac 1q-\frac
1p}u_*(n^{\beta_*}\varphi_*(n))\varphi_*^{-\mu_*}(n).
$$
The further arguments are the same as in the case $p\le q$.

Let $1\le t_*(n)<t_0$. Then
$$
\|f-P^1_nf\|_{Y_q(\Omega)} \underset{\mathfrak{Z}_0}{\lesssim}
2^{-\lambda_*k_*t_0} u_*(2^{k_*t_0}) =2^{-\lambda_*k_*t_0}
u_*(2^{k_*t_0})\left({\rm card}\, \hat J_{t_0}\right)^{\frac
1q-\frac 1p} \stackrel{(\ref{bipf4684gn})}
{\underset{\mathfrak{Z}_0} {\lesssim}}
$$
$$
\lesssim 2^{-\mu_*k_*t_0} u_*(2^{k_*t_0}) \left(2^{\gamma_*k_*t_0}
\psi_*(2^{k_*t_0})\right)^{\frac 1q-\frac 1p} \stackrel
{(\ref{2ll})} {\underset{\mathfrak{Z}_0} {\lesssim}}
$$
$$
\lesssim 2^{-\mu_*k_*t_*(n)} u_*(2^{k_*t_*(n)})
\left(2^{\gamma_*k_*t_*(n)} \psi_*(2^{k_*t_*(n)})\right)^{\frac
1q-\frac 1p} \stackrel{(\ref{2_ks_t})}{\underset{\mathfrak{Z}_0}
{\lesssim}} n^{-\mu_*\beta_*+\frac 1q-\frac
1p}u_*(n^{\beta_*}\varphi_*(n))\varphi_*^{-\mu_*}(n).
$$

{\bf Step 4.} Let us estimate the widths $\vartheta_n$ for $p<q$,
$\hat q>2$. Denote $$t_{**}(n)=\min \{t\in \N:\;
\overline{\nu}_t\ge n^{\hat q/2}\}.$$ Then
\begin{align}
\label{nut_ss} \overline{\nu}_{t_{**}(n)} \underset{\mathfrak
{Z}_0} {\asymp} n^{\hat q/2}, \quad 2^{k_*t_{**}(n)}
\underset{\mathfrak{Z}_0}{\asymp} n^{\beta_* \hat
q/2}\varphi_*(n^{\hat q/2})
\end{align}
(it can be proved similarly as (\ref{2_ks_t})). For sufficiently
large $n\in \N$ we have $t_*(n)+1\le t_{**}(n)-1$. If $t_*(n)+1>
t_{**}(n)-1$, then $n \underset{\mathfrak{Z}_0}{\asymp} n^{\hat
q/2}$. In this case, the desired estimate has been already
obtained at the previous step.

Let $t_*(n)+1\le t_{**}(n)-1$.

{\bf Step 4.1.} Consider $t> t_*(n)$. If $t\ge t_0$, then
$\overline{J}_t=\hat J_t$ by (\ref{ovrl_it_eq_hat_it}). Let
$\tilde P_{t,i}$ be the operator defined at the step 2. We set
\begin{align}
\label{qtf_x} Q_tf(x)=\tilde P_{t,i}f(x) \quad \text{for} \quad
x\in \tilde U_{t,i}, \quad i\in \hat J_t, \quad Q_tf(x)=0\quad
\text{for} \quad x\in \Omega \backslash \tilde U_t.
\end{align}
Notice that if $t<t_0$, then $Q_tf=Q_{t+1}f$.

We claim that
\begin{align}
\label{f_q_t} \begin{array}{c} (f-Q_{t_*(n)+1}f)\chi _{\tilde
U_{t_*(n)+1}}= \sum \limits _{j=t_*(n)+1} ^{t_{**}(n)-1}
(Q_{j+1}f-Q_jf)\chi _{\tilde U_{j+1}} +\\ +\sum \limits
_{t=t_*(n)+1} ^{t_{**}(n)-1} (f-Q_tf) \chi _{G_t}
+(f-Q_{t_{**}(n)}f)\chi _{\tilde U_{t_{**}(n)}}.
\end{array}
\end{align}
Indeed, let $x\in G_t$, $t_*(n)+1\le t\le t_{**}(n)-1$. Then $t\ge
t_0$ (otherwise $G_t=\varnothing$). We have $\chi _{\tilde
U_{t_*(n)+1}}(x)=1$, $\chi _{\tilde U_{t_{**}(n)}}(x)=0$, $\chi
_{\tilde U_{j+1}}(x)=1$ for $j\le t-1$, $\chi _{\tilde
U_{j+1}}(x)=0$ for $j\ge t$. Therefore, the left-hand side of
(\ref{f_q_t}) equals to $f(x)-Q_{t_*(n)+1}f(x)$, as well as the
right-hand side equals to
$$
\sum \limits _{j=t_*(n)+1}^{t-1}
(Q_{j+1}f(x)-Q_jf(x))+f(x)-Q_tf(x) =f(x)-Q_{t_*(n)+1}f(x).
$$
Let $x\in \tilde U_{t_{**}(n)}$. Then $\chi _{\tilde
U_{t_*(n)+1}}(x)=1$, $\chi _{\tilde U_{t_{**}(n)}}(x)=1$, $\chi
_{\tilde U_{j+1}}(x)=1$, $t_*(n)+1\le j\le t_{**}(n)-1$, $\chi
_{G_t}(x)=0$, $t_*(n)+1\le t\le t_{**}(n)-1$. Therefore, the left-hand
side of (\ref{f_q_t}) equals to $f(x)-Q_{t_*(n)+1}f(x)$, and the right-hand
side equals to
$$
\sum \limits _{j=t_*(n)+1} ^{t_{**}(n)-1}
(Q_{j+1}f(x)-Q_jf(x))+(f(x)-Q_{t_{**}(n)}f(x))
=f(x)-Q_{t_*(n)+1}f(x).
$$

If $x\in \Omega \backslash \tilde U_{t_*(n)+1}$, then both parts
of (\ref{f_q_t}) equal to zero.

{\bf Step 4.2.} We set
\begin{align}
\label{m_t_def} m_t=\lceil\log \nu_t\rceil,
\end{align}
$\nu_{t,i}={\rm card}\, {\bf V}({\cal A} _{t,i})$, $i\in \hat
J_t$,
$$
J_{m,t}=\left\{i\in \hat J_t:\; \frac{\nu_{t,i}}{\nu_t}\cdot 2^m
\ge 1\right\}, \quad 0\le m\le m_t,
$$
$$
\overline{m}(i)=\min \{m\in \overline{0, \, m_t}:\; i\in
J_{m,t}\}, \quad i\in \hat J_t
$$
(notice that $J_{m_t,t}=\hat J_t$). Then
$\frac{\nu_{t,i}}{\nu_t}\cdot 2^{\overline{m}(i)}<2$.

We define the function $\Phi_{t,i}:2^{{\bf V}({\cal A}_{t,i})}
\rightarrow \R_+$ by $\Phi_{t,i}({\bf W})={\rm card}\, {\bf W}$,
${\bf W}\subset {\bf V}({\cal A}_{t,i})$. By Lemma
\ref{lemma_o_razb_dereva1}, for any $i\in \hat J_{t}$,
$\overline{m}(i)\le m\le m_t$ there exists a partition $R^i_{m,t}$
of the tree ${\cal A}_{t,i}$ with the following properties:
\begin{enumerate}
\item ${\rm card}\, R^i_{m,t} \underset{\mathfrak{Z}_0}{\lesssim} \left\lceil \frac{\nu_{t,i}}{\nu_t}
\cdot 2^m \right \rceil$;
\item for any tree ${\cal D}\in R^i_{m,t}$
\begin{align}
\label{card_vd} {\rm card}\, {\bf V}({\cal D})
\underset{\mathfrak{Z}_0}{\lesssim} \nu_t \cdot 2^{-m}.
\end{align}
\item for any tree ${\cal D}\in R^i_{m,t}$
$$
{\rm card}\{{\cal D}'\in R^i_{m\pm 1,t}:{\bf V}({\cal D})\cap {\bf
V}({\cal D}')\ne \varnothing\}\underset{\mathfrak{Z}_0}{\lesssim}
1.
$$
\end{enumerate}
Moreover, we may assume that ${\rm card}\,
R^i_{\overline{m}(i),t}=1$.

For $0\le m\le m_t$ we set $G_{m,t} =\cup _{i\in J_{m,t}}
\Omega_{{\cal A}_{t,i}}$,
\begin{align}
\label{t2mt_mlmt_0_def}
T^2_{m,t}=\{\Omega_{{\cal D}}:\, {\cal D}\in R^i_{m,t}, \, i\in
J_{m,t}\}.
\end{align}
Then $T^2_{m,t}$ is a partition of $G_{m,t}$ and
\begin{align}
\label{ct2mt} {\rm card}\, T^2_{m,t}\underset{\mathfrak{Z}_0}
{\lesssim} \frac{2^m\sum \limits _{i\in J_{m,t}}\nu_{t,i}}{\nu_t};
\quad \text{in particular,}\quad {\rm card}\, T^2_{m_t,t}
\stackrel{(\ref{m_t_def})}{\underset{\mathfrak{Z}_0}{\lesssim}}
\nu_t.
\end{align}

By Assumption \ref{sup1} and (\ref{w_s_2}), for any $E\in
T^2_{m,t}$ there exists a linear continuous operator
$P_E:Y_q(E)\rightarrow {\cal P}(E)$ such that
$\|f-Q_tf-P_Ef\|_{Y_q(E)} \underset {\mathfrak{Z}_0} {\lesssim}
2^{-\lambda_*k_*t} u_*(2^{k_*t}) \|f\|_{X_p(E)}$ for any function
$f\in X_p(\Omega)$. Since ${\rm card}\,
R^i_{\overline{m}(i),t}=1$,
$T^2_{\overline{m}(i),t}|_{\Omega_{{\cal
A}_{t,i}}}=\{\Omega_{{\cal A}_{t,i}}\}$. This together with
(\ref{ftp_ti}) and (\ref{qtf_x})  implies that we can take
$P_{\Omega_{{\cal A}_{t,i}}}f=0$. Therefore, for any $m\in \{0, \,
\dots, \, m_t\}$ there exists a linear continuous operator
\begin{align}
\label{p2mt_def_yq} P^2_{m,t}:Y_q(\Omega)\rightarrow {\cal
S}_{T^2_{m,t}}(\Omega)
\end{align}
such that
\begin{align}
\label{p20} P^2_{\overline{m}(i),t}f|_{\Omega_{{\cal A}_{t,i}}}=0,
\quad i\in \hat J_t, \quad f\in X_p(\Omega),
\end{align}
\begin{align}
\label{fp2mt} \|f-Q_tf-P^2_{m,t}f\|_{p,q,T^2_{m,t}}
\underset{\mathfrak{Z}_0}{\lesssim}
2^{-\lambda_*k_*t}u_*(2^{k_*t})\|f\|_{X_p(G_{m,t})}.
\end{align}

Finally, for any $E\in T^2_{m,t}$
\begin{align}
\label{cet2} {\rm card}\, \{E'\in T^2_{m\pm 1,t}:\; {\rm mes}(E\cap
E')>0\} \underset{\mathfrak{Z}_0}{\lesssim} 1.
\end{align}

If $x\in \Omega_{{\cal A}_{t,i}}$, then by the definition of $G_{m,t}$
$$
\sum \limits_{m=0}^{m_t-1} (P^2_{m+1,t}f(x)- P^2_{m,t}f(x))
\chi_{G_{m,t}}(x)= \sum \limits _{m=\overline{m}(i)}^{m_t-1}
(P^2_{m+1,t}f(x)- P^2_{m,t}f(x))\stackrel{(\ref{p20})}{=}
P^2_{m_t,t}f(x).
$$
Hence,
$$
(f-Q_tf)\chi_{G_t} = \sum \limits _{m=0}^{m_t-1} (P^2_{m+1,t}f
-P^2_{m,t}f)\chi_{G_{m,t}} +(f-Q_tf- P^2_{m_t,t} f)\chi_{G_t}.
$$
This together with (\ref{f_q_t}) yields
\begin{align}
\label{fqt2} \begin{array}{c} (f-Q_{t_*(n)+1}f)\chi _{\tilde
U_{t_*(n)+1}}= \sum \limits _{j=t_*(n)+1} ^{t_{**}(n)-1}
(Q_{j+1}f-Q_jf)\chi_{\tilde U_{j+1}} +\\+\sum \limits
_{t=t_*(n)+1} ^{t_{**}(n)-1} \sum \limits _{m=0}^{m_t-1}
(P^2_{m+1,t}f-P^2_{m,t}f)\chi_{G_{m,t}} +\\+\sum \limits
_{t=t_*(n)+1} ^{t_{**}(n)-1} (f-Q_tf-P^2_{m_t,t}f)\chi
_{G_t}+(f-Q_{t_{**}(n)}f)\chi _{\tilde U_{t_{**}(n)}}.
\end{array}
\end{align}

{\bf Step 4.3.} Construct the partition $T^2_{m,t}$ of the set
$G_t$ for $m>m_t$. Let ${\cal D}\in R^i_{m_t,t}$ for some $i\in
\hat J_t$. From (\ref{m_t_def}) and (\ref{card_vd}) it follows
that
\begin{align}
\label{cvd_1} {\rm card}\, {\bf V}({\cal
D})\underset{\mathfrak{Z}_0}{\lesssim} 1.
\end{align}
For each $\xi\in {\bf V}(\cal D)$ we define the partition
$T_{m,t,\xi}:=T_{m-m_t,1}(\hat F(\xi))$ and the operator $P_E$
according to Assumption \ref{sup2}. Recall that $J_{m_t,t}=\hat
J_t$. We set
$$
T^2_{m,t}=\{E\in T_{m,t,\xi}:\; \xi\in {\bf V}({\cal D}), \; {\cal
D}\in R^i_{m_t,t}, \; i\in \hat J_t\},
$$
$P^2_{m,t}f|_E=P_Ef-(Q_tf)|_E$, $E\in T^2_{m,t}$,
$P^2_{m,t}f|_{\Omega \backslash G_t}=0$. Then
\begin{align}
\label{p2mt_s_gt}  \cup _{E\in T^2_{m,t}}E=G_t, \quad P^2_{m,t}:
Y_q(\Omega) \rightarrow {\cal S}_{T^2_{m,t}}(\Omega) \quad
\text{is a linear continuous operator};
\end{align}
by Assumption \ref{sup2}, (\ref{nu_t_k}), (\ref{t2mt_mlmt_0_def}),
(\ref{ct2mt}) and (\ref{cvd_1}),
\begin{align}
\label{ct2_mt} {\rm card}\, T^2_{m,t}\underset{\mathfrak{Z}_0}
{\lesssim} \overline{\nu}_t\cdot 2^{m-m_t}.
\end{align}
Further,
\begin{align}
\label{fp2fmt} \|f-Q_tf-P^2_{m,t}f\|_{p,q,T^2_{m,t}}
\stackrel{(\ref{fpef}),(\ref{mu_ge_lambda}), (\ref{til_w_s_2})}
{\underset{\mathfrak{Z}_0}{\lesssim}}
2^{-\lambda_*k_*t}u_*(2^{k_*t})\cdot
2^{-\delta_*(m-m_t)}\|f\|_{X_p(G_t)},
\end{align}
and for any $E\in T^2_{m,t}$
\begin{align}
\label{cest2} {\rm card}\, \{E'\in T^2_{m\pm 1,t}:\, {\rm mes}(E
\cap E')>0\} \stackrel{(\ref{ceptm})} {\underset{\mathfrak{Z}_0}
{\lesssim}} 1.
\end{align}
From (\ref{fqt2}) we get
\begin{align}
\label{fqt3} \begin{array}{c} (f-Q_{t_*(n)+1}f)\chi _{\tilde
U_{t_*(n)+1}}= \sum \limits _{j=t_*(n)+1} ^{t_{**}(n)-1}
(Q_{j+1}f-Q_jf)\chi _{\tilde U_{j+1}} +\\+\sum \limits
_{t=t_*(n)+1} ^{t_{**}(n)-1} \sum \limits _{m=0}^{m_t-1}
(P^2_{m+1,t}f-P^2_{m,t}f)\chi _{G_{m,t}}+\\+\sum \limits
_{t=t_*(n)+1} ^{t_{**}(n)-1} \sum \limits _{m=m_t}^\infty
(P^2_{m+1,t}f-P^2_{m,t}f) +(f-Q_{t_{**}(n)}f)\chi _{\tilde
U_{t_{**}(n)}}
\end{array}
\end{align}
(by (\ref{fp2fmt}), the series converges in $Y_q(\Omega)$).

{\bf Step 5.} Let us obtain estimates of $\vartheta_n$. Define the
partition $T_n$ by $T_n|_{G_t}=T^1_{m^*_t,n,t}$, $t_0\le t\le
t_*(n)$, $T_n|_{\tilde U_{t_*(n)+1,i}}=\{\tilde U_{t_*(n)+1,i}\}$.
If $t_0>t_*(n)$, then $T_n=\{\Omega\}$. If $t_0\le t_*(n)$, then
\begin{align}
\label{card_tn} \begin{array}{c} \displaystyle {\rm card}\, T_n
\stackrel{(\ref{nu_t_k}),(\ref{ovrl_it_eq_hat_it}),(\ref{2_ks_t}),
(\ref{nt_eq_12mt_eq_2ngkt}), (\ref{ct1mnt})}
{\underset{\mathfrak{Z}_0}{\lesssim}} n+\sum \limits _{t_0\le t\le
t_*(n),\, m_t^*=0} n \cdot 2^{-\varepsilon|t-\hat t(n)|} +\\+\sum
\limits _{t_0\le t\le t_*(n),\; m_t^*>0} \nu_t
\stackrel{(\ref{nu_t_k}),(\ref{2_ks_t})}{\underset
{\varepsilon,\mathfrak{Z}_0}{\lesssim}} n. \end{array}
\end{align}
Set $S_nf|_{G_t}= P^1_{m^*_t,n,t}f|_{G_t}$ for $t_0\le t\le
t_*(n)$, $S_nf|_{\tilde U_{t_*(n)+1}}=Q_{t_*(n)+1}f$. This
together with (\ref{fp1ft}), (\ref{fqt3}) yields that
\begin{align}
\label{sntn} S_nf\in {\cal S}_{T_n}(\Omega),
\end{align}
$f\mapsto S_nf$ is a linear continuous operator in $Y_q(\Omega)$
and
\begin{align}
\label{razl_f} \displaystyle \begin{array}{c} f-S_nf= \sum \limits
_{t=t_0}^{t_*(n)} \sum \limits _{m\ge m^*_t} (P^1_{m+1,n,t}f
-P^1_{m,n,t}f)+\\+ \sum \limits _{t=t_*(n)+1} ^{t_{**}(n)-1}
(Q_{t+1}f-Q_tf)\chi_{\tilde U_{t+1}} + \sum \limits _{t=t_*(n)+1}
^{t_{**}(n)-1} \sum \limits _{m=0}^{m_t-1}
(P^2_{m+1,t}f-P^2_{m,t}f)\chi _{G_{m,t}} +\\+ \sum \limits
_{t=t_*(n)+1} ^{t_{**}(n)-1} \sum \limits _{m=m_t}^\infty
(P^2_{m+1,t}f-P^2_{m,t}f)+(f-Q_{t_{**}(n)}f)\chi _{\tilde
U_{t_{**}(n)}}.
\end{array}
\end{align}
Denote by $\hat T^1_{m,n,t}$ be the partition of $G_t$ that
consists of the sets $E'\cap E''$, $E'\in T^1_{m,n,t}$, $E''\in
T^1_{m+1,n,t}$, ${\rm mes}(E'\cap E'')>0$, $t_0\le t\le t_*(n)$,
$m\ge m^*_t$. Let $\hat T^2_{m,t}$ be the partition of $G_{m,t}$
that consists of the sets $E'\cap E''$, $E'\in T^2_{m,t}$, $E''\in
T^2_{m+1,t}$, ${\rm mes}(E'\cap E'')>0$, $t_*(n)+1\le t\le
t_{**}(n)-1$, $m\ge 0$ ($G_{m,t}:=G_t$ for $m\ge m_t$; see
(\ref{p2mt_s_gt})). Let $\hat T^3_t=\{\tilde U_{t+1,i}\}_{i\in
\overline{J}_{t+1}}$, $t_*(n)+1\le t\le t_{**}(n)-1$. Then
\begin{align}
\label{pitm_st} P^1_{m+1,n,t}f-P^1_{m,n,t}f
\stackrel{(\ref{p1mnt})}{\in} {\cal S} _{\hat
T^1_{m,n,t}}(\Omega), \quad (P^2_{m+1,t}f-P^2_{m,t}f)\chi
_{G_{m,t}} \stackrel{(\ref{p2mt_def_yq}),(\ref{p2mt_s_gt})}{\in}
{\cal S} _{\hat T^2_{m,t}}(\Omega),
\end{align}
\begin{align}
\label{qtf} (Q_{t+1}f-Q_tf)\chi _{\tilde U_{t+1}}
\stackrel{(\ref{qtf_x})}{\in} {\cal S} _{\hat T^3_t}(\Omega),
\end{align}
for any function $f\in BX_p(\Omega)$
\begin{align}
\label{pp1} \|P^1_{m+1,n,t}f-P^1_{m,n,t}f\|_{p,q,\hat T^1_{m,n,t}}
\stackrel{(\ref{mu_ge_lambda}),(\ref{fp1ft}),(\ref{card_e})}
{\underset{\mathfrak{Z}_0}{\lesssim}}
2^{-\lambda_*k_*t}u_*(2^{k_*t}) (2^{m-m_t^*}n_t)^{-\delta_*},
\end{align}
\begin{align}
\label{pp21} \|P^2_{m+1,t}f-P^2_{m,t}f\|_{p,q,\hat T^2_{m,t}}
\stackrel{(\ref{fp2mt}),(\ref{cet2})}{\underset{\mathfrak{Z}_0}
{\lesssim}} 2^{-\lambda_*k_*t}u_*(2^{k_*t}), \quad m< m_t,
\end{align}
\begin{align}
\label{pp22} \|P^2_{m+1,t}f-P^2_{m,t}f\|_{p,q,\hat T^2_{m,t}}
\stackrel{(\ref{fp2fmt}),(\ref{cest2})} {\underset{\mathfrak{Z}_0}
{\lesssim}} 2^{-\lambda_*k_*t}u_*(2^{k_*t})\cdot
2^{-\delta_*(m-m_t)}, \quad m\ge m_t,
\end{align}
\begin{align}
\label{pp3} \|Q_{t+1}f-Q_tf\| _{p,q, \hat T^3_t}
\stackrel{(\ref{ftp_ti}),(\ref{qtf_x})}{\underset{\mathfrak{Z}_0}
{\lesssim}} 2^{-\lambda_*k_*t}u_*(2^{k_*t}),
\end{align}
\begin{align}
\label{pp3_11} \|f-Q_tf\| _{Y_q(\tilde U_t)}
\stackrel{(\ref{ftp_ti}),(\ref{qtf_x})}{\underset{\mathfrak{Z}_0}
{\lesssim}} 2^{-\lambda_*k_*t}u_*(2^{k_*t}),
\end{align}
and there exists $C=C(\mathfrak{Z}_0)>0$ such that
\begin{align}
\label{dim_stmn1}  s^1_{m,n,t}:=\dim {\cal S}_{\hat
T^1_{m,n,t}}(\Omega) \stackrel{(\ref{ct1mnt}),(\ref{card_e})}{\le}
C\cdot 2^{m-\varepsilon|t-\hat t(n)|}n, \quad t_0\le t\le t_*(n),
\; \; m\ge m^*_t,
\end{align}
\begin{align}
\label{dim_stmn12} s^2_{m,t}:=\dim {\cal S}_{\hat
T^2_{m,t}}(\Omega)\stackrel{(\ref{ct2mt}),(\ref{cet2})}{\le} C\cdot 2^m, \quad
t_*(n)+1\le t\le t_{**}(n)-1, \;\; m\le m_t,
\end{align}
\begin{align}
\label{dim_stmn22} s^2_{m,t}:=\dim {\cal S}_{\hat
T^2_{m,t}}(\Omega) \stackrel{(\ref{ct2_mt}),(\ref{cest2})}{\le}
C\cdot 2^{m-m_t}\overline{\nu}_t, \quad t_*(n)+1\le t\le
t_{**}(n)-1, \;\; m> m_t,
\end{align}
\begin{align}
\label{ct3} s^3_t:=\dim {\cal S}_{\hat T^3_t}(\Omega)
\stackrel{(\ref{nu_t_k}),(\ref{ovrl_it_eq_hat_it})}{\le} C\cdot
2^{\gamma_*k_*t} \psi_*(2^{k_*t}), \quad t_*(n)+1\le t\le
t_{**}(n)-1.
\end{align}

By Lemma \ref{oper_a}, there exist linear isomrphisms
$A^1_{m,n,t}:{\cal S} _{\hat T^1_{m,n,t}}(\Omega) \rightarrow
\R^{s^1_{m,n,t}}$, $A^2_{m,t}:{\cal S} _{\hat T^2_{m,t}}(\Omega)
\rightarrow \R^{s^2_{m,t}}$, $A^3_t:{\cal S} _{\hat T^3_t}(\Omega)
\rightarrow \R^{s^3_t}$ such that
\begin{align}
\label{a1} \|A^1_{m,n,t}\|_{Y_{p,q,\hat T^1_{m,n,t}}(G_t)
\rightarrow l_p^{s^1_{m,n,t}}} \underset{\mathfrak{Z}_0}{\lesssim}
1, \quad \|(A^1_{m,n,t}) ^{-1}\| _{l_q^{s^1_{m,n,t}}\rightarrow
Y_q(G_t)} \underset{\mathfrak{Z}_0}{\lesssim} 1,
\end{align}
\begin{align}
\label{a2} \|A^2_{m,t}\|_{Y_{p,q,\hat T^2_{m,t}}(G_{m,t}) \rightarrow
l_p^{s^2_{m,t}}} \underset{\mathfrak{Z}_0}{\lesssim} 1, \quad
\|(A^2_{m,t}) ^{-1}\| _{l_q^{s^2_{m,t}}\rightarrow Y_q(G_{m,t})}
\underset{\mathfrak{Z}_0}{\lesssim} 1,
\end{align}
\begin{align}
\label{a3} \|A^3_t\|_{Y_{p,q,\hat T^3_t}(\tilde U_{t+1})
\rightarrow l_p^{s^3_t}} \underset{\mathfrak{Z}_0}{\lesssim} 1,
\quad \|(A^3_t) ^{-1}\| _{l_q^{s^3_t}\rightarrow Y_q(\tilde
U_{t+1})} \underset{\mathfrak{Z}_0}{\lesssim} 1.
\end{align}
Given $k^1_{m,n,t}\in \Z_+$, let $E^1_{m,n,t}:\R^{s^1_{m,n,t}}
\rightarrow \R^{s^1_{m,n,t}}$ be the extremal mapping for
$\gamma^1_{m,n,t}:=\vartheta _{k^1_{m,n,t}}(B_p^{s^1_{m,n,t}}, \,
l_q^{s^1_{m,n,t}})$ ($t_0\le t\le t_*(n)$, $m\ge m^*_t$); given
$k^2_{m,t}\in \Z_+$, let $E^2_{m,t}:\R^{s^2_{m,t}} \rightarrow
\R^{s^2_{m,t}}$ be the extremal mapping for $\gamma^2_{m,t}
:=\vartheta _{k^2_{m,t}}(B_p^{s^2_{m,t}}, \, l_q^{s^2_{m,t}})$
($t_*(n)<t<t_{**}(n)$, $m\in \Z_+$); given $k^3_t\in \Z_+$, let
$E^3_t:\R^{s^3_t} \rightarrow \R^{s^3_t}$ be the extremal mapping
for $\gamma^3_t :=\vartheta_{k^3_t}(B_p^{s^3_t}, \, l_q^{s^3_t})$
($t_*(n)<t<t_{**}(n)$). Suppose that
\begin{align}
\label{sll} \sum \limits _{t=t_0}^{t_*(n)} \sum \limits _{m\ge
m_t^*} k^1_{m,n,t} + \sum \limits _{t=t_*(n)+1}^{t_{**}(n)-1} \sum
\limits _{m\in \Z_+} k^2_{m,t} +\sum \limits
_{t=t_*(n)+1}^{t_{**}(n)-1} k^3_t \le C_1n,
\end{align}
where $C_1>0$ does not depend on $n$.

Denote
$$
P_*f= S_nf+\sum \limits _{t=t_0}^{t_*(n)} \sum \limits _{m\ge
m_t^*}(A^1_{m,n,t})^{-1}E^1_{m,n,t}A^1_{m,n,t}(P^1_{m+1,n,t}f-P^1_{m,n,t}f)+
$$
$$
+\sum \limits _{t=t_*(n)+1}^{t_{**}(n)-1}\sum \limits _{m\in \Z_+}
(A^2_{m,t})^{-1}E^2_{m,t}A^2_{m,t}[(P^2_{m+1,t}f-P^2_{m,t}f)\chi
_{G_{m,t}}]+
$$
$$
+\sum \limits _{t=t_*(n)+1}^{t_{**}(n)-1} (A^3_t)^{-1} E^3_t
A^3_t[(Q_{t+1}f-Q_tf)\chi_{\tilde U_{t+1}}].
$$
The image of $P_*$ is contained in the subspace $Z\subset
Y_q(\Omega)$ of dimension
$\overline{k}_n\stackrel{(\ref{card_tn}), (\ref{sntn}),
(\ref{sll})}{\underset{\mathfrak{Z}_0,\varepsilon,C_1} {\lesssim}}
n$. Observe that if the mappings $E^1_{m,n,t}$, $E^2_{m,t}$ and
$E^3_t$ are linear, then $P_*$ is a linear continuous operator in
$Y_q(\Omega)$.

In order to estimate Kolmogorov and linear widths, it is sufficient to
obtain an upper bound of $\|f-P_*f\| _{Y_q(\Omega)}$ for $f\in BX_p(\Omega)$.
In estimating Gelfand widths, we obtain an upper bound of $\|f\|_{Y_q(\Omega)}$
for $f\in BX_p(\Omega)$ such that
\begin{align}
\label{snf} S_nf=0, \quad E^1_{m,n,t}A^1_{m,n,t} (P^1_{m+1,n,t}f-
P^1_{m,n,t}f)=0, \quad t_0\le t\le t_*(n), \; m\ge m^*_t,
\end{align}
\begin{align}
\label{e2mt} E^2_{m,t}A^2_{m,t} [(P^2_{m+1,t}f
-P^2_{m,t}f)\chi_{G_{m,t}}]=0, \quad t_*(n)<t<t_{**}(n), \; m\in
\Z_+,
\end{align}
\begin{align}
\label{e3t} E^3_tA^3_t [(Q_{t+1}f -Q_tf)\chi_{\tilde U_{t+1}}]=0,
\quad t_*(n)<t<t_{**}(n),
\end{align}
and take into account (\ref{card_tn}), (\ref{sntn}) and
(\ref{sll}) (indeed, the subspace of functions satisfying
(\ref{snf}), (\ref{e2mt}) and (\ref{e3t}) is closed in
$Y_q(\Omega)$ and its codimension equals to
$C_2n\underset{\mathfrak{Z}_0,\varepsilon,C_1} {\lesssim} n$).

We set
\begin{align}
\label{til_s1} \tilde s^1_{m,n,t}=\lceil C\cdot
2^{m-\varepsilon|t-\hat t(n)|}n\rceil, \quad \quad t_0\le t\le
t_*(n), \; \; m\ge m^*_t,
\end{align}
\begin{align}
\label{til_s21} \tilde s^2_{m,t}=\lceil C\cdot 2^m\rceil, \quad
\quad t_*(n)+1\le t\le t_{**}(n)-1, \;\; m\le m_t,
\end{align}
\begin{align}
\label{til_s22} \tilde s^2_{m,t}=\lceil C\cdot
2^{m-m_t}\overline{\nu}_t\rceil, \quad \quad t_*(n)+1\le t\le
t_{**}(n)-1, \;\; m> m_t,
\end{align}
\begin{align}
\label{til_s3} \tilde s^3_t=\lceil C\cdot 2^{\gamma_*k_*t}
\psi_*(2^{k_*t})\rceil, \quad \quad t_*(n)+1\le t\le t_{**}(n)-1.
\end{align}

Let $\tilde\gamma ^1_{m,n,t}:= \vartheta_{k^1_{m,n,t}}
(B_p^{\tilde s^1_{m,n,t}}, \, l_q^{\tilde s^1_{m,n,t}})$ ($t_0\le
t\le t_*(n)$, $m\ge m_t^*$), $\tilde\gamma ^2_{m,t} :=
\vartheta_{k^2_{m,t}}(B_p^{\tilde s^2_{m,t}}, \, l_q^{\tilde
s^2_{m,t}})$ ($t_*(n)<t<t_{**}(n)$, $m\in \Z_+$), $\tilde\gamma
^3_t := \vartheta_{k^3_t}(B_p^{\tilde s^3_t}, \, l_q^{\tilde
s^3_t})$ ($t_*(n)<t<t_{**}(n)$). Then
\begin{align}
\label{gg} \gamma ^1_{m,n,t}\stackrel {(\ref{dim_stmn1}),
(\ref{til_s1})} {\le} \tilde\gamma ^1_{m,n,t}, \quad t_0\le t\le
t_*(n), \;\; m\ge m_t^*,
\end{align}
\begin{align}
\label{gg2} \gamma ^2_{m,t} \stackrel{(\ref{dim_stmn12}),
(\ref{dim_stmn22}), (\ref{til_s21}), (\ref{til_s22})}
{\le}\tilde\gamma ^2_{m,t}, \quad t_*(n)<t<t_{**}(n), \quad m\in
\Z_+,
\end{align}
\begin{align}
\label{gg3} \gamma^3_t \stackrel{(\ref{ct3}), (\ref{til_s3})}{\le
} \tilde \gamma^3_t, \quad t_*(n)<t<t_{**}(n).
\end{align}
Let $f\in X_p(\Omega)$, $\|f\|_{X_p(\Omega)}\le 1$. In estimating
Gelfand widths, we suppose that
(\ref{snf}), (\ref{e2mt}) and (\ref{e3t}) hold (in particular,
$P_*f=0$).

From the definition of $P_*f$, (\ref{razl_f}), (\ref{pp1}),
(\ref{pp21}), (\ref{pp22}), (\ref{pp3}), (\ref{pp3_11}),
(\ref{a1}), (\ref{a2}), (\ref{a3}), (\ref{gg}), (\ref{gg2}) and
(\ref{gg3}) it follows that
$$
\| f-P_*f\| _{Y_q(\Omega)} \underset{\mathfrak{Z}_0}{\lesssim}
2^{-\lambda_*k_*t_{**}(n)}u_*(2^{k_*t_{**}(n)}) +
$$
$$
+\sum \limits _{t=t_0}^{t_*(n)} \sum \limits _{m\ge m_t^*}
2^{-\lambda_*k_*t}u_*(2^{k_*t}) (2^{m-m_t^*}n_t) ^{-\delta _*}
\tilde \gamma^1_{m,n,t}+ \sum \limits _{t=t_*(n)+1}^{t_{**}(n)-1}
\sum \limits _{m=0}^{m_t-1} 2^{-\lambda_*k_*t}u_*(2^{k_*t}) \tilde
\gamma^2_{m,t} +
$$
$$
+\sum \limits _{t=t_*(n)+1}^{t_{**}(n)-1} \sum \limits
_{m=m_t}^\infty 2^{-\lambda_*k_*t}u_*(2^{k_*t}) \cdot
2^{-\delta_*(m-m_t)} \tilde \gamma^2_{m,t} +
$$
$$
+ \sum \limits_{t=t_*(n)+1}^{t_{**}(n)-1}
2^{-\lambda_*k_*t}u_*(2^{k_*t}) \tilde \gamma^3_t
\stackrel{(\ref{2mmtnut}),(\ref{nut_ss})}{\underset{\mathfrak{Z}_0}{\lesssim}}
n^{-\frac{\lambda_*\beta_*\hat q}{2}}\varphi_*^{-\lambda_*}(n^{\frac{\hat
q}{2}})u_*(n^{\beta_*\hat
q/2}\varphi_*(n^{\hat q/2}))+
$$
$$
+ \sum \limits _{t=t_0}^{t_*(n)} \sum \limits _{m\ge m_t^*}
2^{-\lambda_*k_*t} u_*(2^{k_*t}) \cdot 2^{-m\delta_*}\cdot
n^{-\delta_*}\cdot 2^{\gamma_*\delta_*k_*t+\varepsilon
\delta_*|t-\hat t(n)|}\psi_*^{\delta_*}(2^{k_*t})\tilde \gamma
^1_{m,n,t}+
$$
$$
+\sum \limits _{t=t_*(n)+1}^{t_{**}(n)-1} \sum \limits
_{m=0}^{m_t-1} 2^{-\lambda_*k_*t} u_*(2^{k_*t}) \tilde
\gamma^2_{m,t} +
$$
$$
+\sum \limits _{t=t_*(n)+1}^{t_{**}(n)-1} \sum
\limits _{m=m_t}^\infty 2^{-\lambda_*k_*t} u_*(2^{k_*t}) \cdot
2^{-\delta_*(m-m_t)} \tilde \gamma^2_{m,t} + \sum \limits
_{t=t_*(n)+1}^{t_{**}(n)-1} 2^{-\lambda_*k_*t} u_*(2^{k_*t})
\tilde \gamma^3_t.
$$
Similarly as in \cite{vas_bes} it can be proved that there are
$k^1_{m,n,t}$ and $k^2_{m,t}$ such that
$$
\sum \limits _{t=t_0}^{t_*(n)} \sum \limits _{m\ge m_t^*}
k^1_{m,n,t}+ \sum \limits _{t=t_*(n)+1}^{t_{**}(n)-1} \sum \limits
_{m=m_t}^\infty
k^2_{m,t}\underset{\mathfrak{Z}_0,\varepsilon}{\lesssim} n,
$$
$$
\sum \limits _{t=t_0}^{t_*(n)} \sum \limits _{m\ge m_t^*}
2^{-\lambda_*k_*t} u_*(2^{k_*t}) \cdot 2^{-m\delta_*}\cdot
n^{-\delta_*}\cdot 2^{\gamma_*\delta_*k_*t+\varepsilon
\delta_*|t-\hat t(n)|}\psi_*^{\delta_*}(2^{k_*t})\tilde \gamma
_{m,n,t}^1+
$$
$$
+\sum \limits _{t=t_*(n)+1}^{t_{**}(n)-1} \sum \limits
_{m=m_t}^\infty 2^{-\lambda_*k_*t}u_*(2^{k_*t}) \cdot
2^{-\delta_*(m-m_t)} \tilde \gamma^2_{m,t}
\underset{\mathfrak{Z}_0,\varepsilon}{\lesssim}
n^{-\theta_{j_*}}\sigma_{j_*}(n).
$$
At this step we choose $\hat t(n)$ and an upper bound for
$\varepsilon$.

Thus, it remains to obtain sequences $k^2_{m,t}$
($t_*(n)<t<t_{**}(n)$, $m< m_t$) and $k^3_t$
($t_*(n)<t<t_{**}(n)$) such that
\begin{align}
\label{s_l_ttn1_l2mt_l3t} \sum \limits _{t=t_*(n)+1}^{t_{**}(n)-1}
\sum \limits _{m=0}^{m_t-1} k^2_{m,t}\underset{\varepsilon}
{\lesssim} n, \quad \sum \limits _{t=t_*(n)+1}^{t_{**}(n)-1} k^3_t
\underset{\varepsilon}{\lesssim} n,
\end{align}
\begin{align}
\label{sl_ttsn1} \sum \limits _{t=t_*(n)+1}^{t_{**}(n)-1} \sum
\limits _{m=0}^{m_t-1} 2^{-\lambda_*k_*t} u_*(2^{k_*t}) \tilde
\gamma^2_{m,t}\underset{\varepsilon,\mathfrak{Z}_0}{\lesssim}
n^{-\theta_{j_*}}\sigma_{j_*}(n),
\end{align}
\begin{align}
\label{2kstg3ttj} \sum \limits _{t=t_*(n)+1}^{t_{**}(n)-1}
2^{-\lambda_*k_*t} u_*(2^{k_*t}) \tilde \gamma^3_t
\underset{\varepsilon, \mathfrak{Z}_0}{\lesssim}
n^{-\theta_{j_*}}\sigma _{j_*}(n).
\end{align}

Let $t^1_n=t_*(n)+1$ or $t_n^1=t_{**}(n)-1$ (the choice will be
made later). For $m\le m_t$ we set
$$
k^2_{m,t}=\left\{ \begin{array}{l} \left\lceil n\cdot
2^{-\varepsilon |t-t^1_n|
-\varepsilon(m_t-m)}\right\rceil , \quad \text{if}\;\;
\left\lceil n\cdot 2^{-\varepsilon |t-t^1_n|
-\varepsilon(m_t-m)}\right\rceil \le \frac 12 \tilde s^2_{m,t}, \\
\tilde s^2_{m,t}, \quad \text{otherwise},\end{array}\right .
$$
$$
k^3_t=\left\{ \begin{array}{l} \left\lceil n\cdot 2^{-\varepsilon
|t-t^1_n|}\right\rceil, \quad \text{if}\;\; \left\lceil n\cdot
2^{-\varepsilon |t-t^1_n|}\right\rceil \le \frac 12 \tilde s^3_t, \\
\tilde s^3_t, \quad \text{otherwise}.\end{array}\right .
$$
Then, for sufficiently small $\varepsilon>0$, we have $ \left\lceil
n\cdot 2^{-\varepsilon |t-t^1_n| -\varepsilon(m_t-m)}\right\rceil
\asymp   n\cdot 2^{-\varepsilon |t-t^1_n| -\varepsilon(m_t-m)}$, $
\left\lceil n\cdot 2^{-\varepsilon |t-t^1_n|}\right\rceil \asymp
n\cdot 2^{-\varepsilon |t-t^1_n|}$, which implies
(\ref{s_l_ttn1_l2mt_l3t}). Apply Theorem \ref{glus_trm}.
Set $\lambda_{pq}=\min \left\{\frac{\frac 1p-\frac 1q}{\frac
12-\frac {1}{\hat q}}, \, 1\right\}$. Then for small
$\varepsilon>0$
$$
\sum \limits _{t=t_*(n)+1}^{t_{**}(n)-1} \sum \limits
_{m=0}^{m_t-1} 2^{-\lambda_*k_*t}u_*(2^{k_*t}) \tilde
\gamma^2_{m,t}\stackrel{(\ref{til_s21})}{\underset{\mathfrak{Z}
_0}{\lesssim}}
$$
$$
\lesssim \sum \limits _{t=t_*(n)+1}^{t_{**}(n)-1} \sum \limits
_{m=0}^{m_t-1} 2^{-\lambda_*k_*t} u_*(2^{k_*t})\cdot
2^{\frac{\lambda_{pq}m}{\hat q}}\cdot n^{-\frac
{\lambda_{pq}}{2}}\cdot 2^{\frac
{\lambda_{pq}\varepsilon}{2}(|t-t^1_n|+m_t-m)}
\underset{\mathfrak{Z} _0}{\lesssim}
$$
$$
\lesssim \sum \limits _{t=t_*(n)+1}^{t_{**}(n)-1}
2^{-\lambda_*k_*t} u_*(2^{k_*t})\cdot
2^{\frac{\lambda_{pq}m_t}{\hat q}}\cdot n^{-\frac
{\lambda_{pq}}{2}}\cdot 2^{\frac{\lambda_{pq}\varepsilon}{2}
|t-t^1_n|} \stackrel{(\ref{nu_t_k}), (\ref{m_t_def})}
{\underset{\mathfrak{Z} _0}{\lesssim}}
$$
$$
\lesssim\sum \limits _{t=t_*(n)+1}^{t_{**}(n)-1}2^{-\lambda_*k_*t}
u_*(2^{k_*t})\cdot \left(2^{\gamma_*k_*t}\psi_*(2^{k_*t})\right)
^\frac{\lambda_{pq}}{\hat q} \cdot n^{-\frac
{\lambda_{pq}}{2}}\cdot 2^{\frac {\lambda_{pq}\varepsilon}{2}
|t-t^1_n|}=:S,
$$
$$
\sum \limits _{t=t_*(n)+1}^{t_{**}(n)-1} 2^{-\lambda_*k_*t}
u_*(2^{k_*t}) \tilde \gamma^3_t \stackrel{(\ref{til_s3})}
{\underset{\mathfrak{Z}_0}{\lesssim}}
$$
$$
\lesssim \sum \limits _{t=t_*(n)+1} ^{t_{**}(n)-1}
2^{-\lambda_*k_*t} u_*(2^{k_*t})\cdot
\left(2^{\gamma_*k_*t}\psi_*(2^{k_*t})\right)^\frac{\lambda_{pq}}{\hat
q} \cdot n^{-\frac {\lambda_{pq}}{2}}\cdot
2^{\frac{\lambda_{pq}\varepsilon}{2} |t-t^1_n|}=S.
$$
We set $t^1_n=t_*(n)+1$ if $\frac{\gamma_*\lambda_{pq}}{\hat
q}-\lambda_*\le 0$ (i.e., $\lambda_*\beta_*\ge
\frac{\lambda_{pq}}{\hat q}$), and we take $t^1_n=t_{**}(n)-1$ if
$\frac{\lambda_{pq}\gamma_*}{\hat q}-\lambda _*> 0$ (i.e.,
$\lambda_* \beta_*< \frac{\lambda_{pq}}{\hat q}$). Apply Lemma
\ref{sum_lem}. If $\lambda_*\beta_*>\frac{\lambda_{pq}}{\hat
q}$, then for sufficiently small $\varepsilon>0$
$$
S\underset{\mathfrak{Z}_0,\varepsilon}{\lesssim}
2^{-\lambda_*k_*t_*(n)} u_*(2^{k_*t_*(n)})
\left(2^{\gamma_*k_*t_*(n)} \psi_*(2^{k_*t_*(n)})
\right)^\frac{\lambda_{pq}}{\hat q} \cdot n^{-\frac{\lambda
_{pq}}{2}} \stackrel{(\ref{nu_t_k}),(\ref{2_ks_t})} {\underset{\mathfrak{Z}_0}
{\lesssim}}$$$$\lesssim n^{-\lambda_*\beta_*-\min\left\{\frac 1p-\frac 1q, \,
\frac 12-\frac{1}{\hat q}\right\}} \sigma_3(n);
$$
if $\lambda_*\beta_*< \frac{\lambda_{pq}}{\hat q}$, then for small
$\varepsilon>0$
$$
S\underset{\mathfrak{Z}_0,\varepsilon }{\lesssim}
2^{-\lambda_*k_*t_{**}(n)} u_*(2^{k_*t_{**}(n)})
\left(2^{\gamma_*k_*t_{**}(n)} \psi_*(2^{k_*t_{**}(n)})\right)
^\frac{\lambda_{pq}}{\hat q} \cdot n^{-\frac {\lambda_{pq}}{2}}
\stackrel{(\ref{nut_ss})}{\underset{\mathfrak{Z}_0}{\lesssim}}
n^{-\frac{\hat q \lambda_*\beta_*}{2}} \sigma_4(n).
$$
If $\lambda_*\beta_*= \frac{\lambda_{pq}}{\hat q}$, then $\theta_3=
\theta_4$, and by the conditions of the theorem  we have $j_*\in \{1, \, 2\}$. By
(\ref{nut_ss}) and Lemma \ref{sum_lem}, there exists $c=
c(\mathfrak{Z}_0)>0$ such that $t_{**}(n)\le c(1+\log n)$.
Therefore, there is $\tilde c=\tilde c(\mathfrak{Z}_0)>0$ such that
for any $\varepsilon >0$
$$
S\underset{\varepsilon,\mathfrak{Z}_0}{\lesssim}
n^{-\lambda_*\beta_*+\frac{\lambda_{pq}}{\hat q}-\frac
{\lambda_{pq}}{2}+\tilde c\varepsilon}=n^{-\theta_3+ \tilde
c\varepsilon}
$$
(it follows from (\ref{nut_ss}) and Lemma \ref{sum_lem}).
If $\varepsilon$ is sufficiently small, we get
$S\underset{\mathfrak{Z}_0}{\lesssim} n^{-\theta_{j_*}}$. This completes
the proof of (\ref{sl_ttsn1}) and (\ref{2kstg3ttj}).
\end{proof}

\begin{Rem}
\label{diskr_case} Suppose that Assumptions \ref{sup1} and
\ref{sup3} hold, and instead of Assumption \ref{sup2} the
following condition holds: for any $\xi\in {\bf V}({\cal A})$ the
set $F(\xi)$ is the atom of measure ${\rm mes}$. Denote
$\mathfrak{Z}_0=(p, \, q,\, r_0, \, w_*, \, k_*, \, \lambda_*, \,
\mu_*,\, \gamma_*, \, \psi_*,\, u_*,\, c_1, \, c_2, \, c_3)$. Then
the assertion of Theorem \ref{main_abstr_th} holds with
$\delta_*=+\infty$; i.e.
\begin{align}
\label{vrth_n_pleq00000} \vartheta_n(BX_p(\Omega), \, Y_q(\Omega))
\underset{\mathfrak{Z}_0} {\lesssim}
n^{-\lambda_*\beta_*}u_*(n^{\beta_*} \varphi_*(n))
\varphi_*^{-\lambda_*}(n), \quad p\le q,
\end{align}
\begin{align}
\label{vrth_n_pgq00000} \vartheta_n(BX_p(\Omega), \, Y_q(\Omega))
\underset{\mathfrak{Z}_0} {\lesssim} n^{-\mu_*\beta_*+\frac
1q-\frac 1p}u_*(n^{\beta_*} \varphi_*(n)) \varphi_*^{-\mu_*}(n),
\quad p>q;
\end{align}
if $p<q$, $\hat q>2$, $\theta_3 =\lambda_*\beta_* +\min \left\{
\frac 12- \frac{1}{\hat q}, \, \frac 1p-\frac 1q\right\}$,
$\theta_4 =\frac{\lambda_*\beta_* \hat q}{2}$,
$\sigma_3(n)=u_*(n^{\beta_*}\varphi_*(n))\varphi_*^{-\lambda_*}(n)$,
$\sigma_4(n)= \sigma_3(n^{\frac{\hat q}{2}})$ and $\theta_3\ne
\theta_4$, $\theta_{j_*}=\min \{\theta_3, \, \theta_4\}$, $j_*\in
\{3, \, 4\}$, then
$$
\vartheta_n(BX_p(\Omega), \, Y_q(\Omega))
\underset{\mathfrak{Z}_0} {\lesssim} n^{-\theta_{j_*}}
\sigma_{j_*}(n).
$$
\end{Rem}

The following result gives the lower estimate of $n$-widths.
\begin{Lem}
\label{low_est} Let $\Omega$ be a measure space, let $G_1, \,
\dots, \, G_m\subset \Omega$ be pairwise non-overlapping measurable
subsets, let $\psi_1, \, \dots, \, \psi_m\in X_p(\Omega)$, ${\rm
supp}\, \psi_j\subset G_j$,  $\|\psi_j\|_{X_p(\Omega)}=1$,
\begin{align}
\label{psi_j_yq_ge_m}
\|\psi_j\|_{Y_q(\Omega)}\ge M, \quad 1\le j\le m.
\end{align}
Then
$$
\vartheta_n(BX_p(\Omega), \, Y_q(\Omega))\ge M\vartheta_n(B^m_p,
\, l^m_q).
$$
If $p>q$ and if, instead of (\ref{psi_j_yq_ge_m}), the
inequalities
\begin{align}
\label{psi_j_yq_ge_m1} \|\psi_j\|_{Y_q(\Omega)}\ge  M_j, \quad
1\le j\le m,
\end{align}
hold with $M_1\ge M_2\ge\dots \ge M_m>0$, then
\begin{align}
\label{tn_low_est_diag} \vartheta_n(BX_p(\Omega), \,
Y_q(\Omega))\ge \left( \sum \limits _{j=n+1}^m
M_j^{\frac{pq}{p-q}} \right)^{\frac 1q-\frac 1p}.
\end{align}
\end{Lem}
The proof is the same as in
\cite{peak_lim}. In order to prove (\ref{tn_low_est_diag})
we apply (\ref{diag_pietsh}).

\section{Weighted Sobolev classes on a John domain}

Let $\Theta\subset \Xi\left(\left[-\frac 12, \, \frac
12\right]^d\right)$ be a family of pairwise non-overlapping
cubes.

\begin{Def}
Let ${\cal G}$ be a graph, and let $F:{\bf V}({\cal G})
\rightarrow \Theta$ be a one-to-one mapping. We say that $F$ is
consistent with the structure of the graph ${\cal G}$ if the
following condition holds: for any adjacent vertices $\xi'$, $\xi''\in
{\bf V}({\cal G})$ the set $\Gamma _{\xi',\xi''} :=F(\xi')\cap F(\xi'')$
has dimension $d-1$.
\end{Def}

Let $({\cal T}, \, \xi_*)$ be a tree, and let $F:{\bf V}({\cal T})
\rightarrow \Theta$  be a one-to-one mapping consistent with the
structure of the tree ${\cal T}$. For any adjacent vertices
$\xi'$, $\xi''$, we set $\mathaccent '27 \Gamma
_{\xi',\xi''}=\inter _{d-1}\Gamma _{\xi',\xi''}$, and for each
subtree ${\cal T}'$ of ${\cal T}$, we put
$$
\Omega _{{\cal T'},F}=\left(\cup _{\xi\in {\bf V}({\cal
T'})}\inter F(\xi)\right) \cup \left(\cup _{(\xi',\xi'')\in {\bf
E}({\cal T'})}\mathaccent '27 \Gamma _{\xi',\xi''}\right).
$$

Let $\Omega\in {\bf FC}(a)$. Suppose that conditions (\ref{def_h}),
(\ref{yty}), (\ref{ghi_g0}), (\ref{psi_cond}), (\ref{muck}),
(\ref{beta}), (\ref{phi_g}), (\ref{ll}), (\ref{g0ag}) hold.

Let $\Theta(\Omega)$ be a Whitney covering of $\Omega$
(see Theorem \ref{whitney}). The following lemma is
proved in \cite{vas_john}.
\begin{Lem}
\label{cor_omega_t} Let $a>0$, $\Omega\subset \left[-\frac 12, \,
\frac 12\right]^d$, $\Omega\in {\bf FC}(a)$. Then there exist a
tree ${\cal T}$ and a one-to-one mapping $F:{\bf V}({\cal T})
\rightarrow \Theta(\Omega)$ consistent with the structure of
${\cal T}$ and a number $b_*=b_*(a, \, d)>0$ such that
$\Omega_{{\cal T}',F}\in {\bf FC}(b_*)$ for any subtree ${\cal
T}'$ of ${\cal T}$.
\end{Lem}

Let $\{{\cal D}_{j,i}\}_{j\ge j_*, \, i\in \tilde I_j}$ be a
partition of the tree ${\cal T}$ into non-empty subtrees. We say
that this partition is regular if $j'=j+1$ for any $j$, $j'\ge
j_*$, $i\in \tilde I_j$, $i'\in \tilde I_{j'}$ such that ${\cal
D}_{j',i'}$ succeeds the tree ${\cal D}_{j,i}$.

In \cite{vas_vl_raspr} the number $\overline{s}=\overline{s}(a, \,
d) \in \N$ is defined and the regular partition $\{({\cal D}_{j,i}, \,
\hat \xi_{j,i})\}_{j\ge j_*, \, i\in \tilde I_j}$ of the tree ${\cal
T}$ is constructed. This partition satisfies the following conditions:
\begin{enumerate}
\item for any $j\ge j_*$, $i\in \tilde I_j$
\begin{align}
\label{diam_dj} {\rm diam}\, \Omega_{{\cal
D}_{j,i},F}{\underset{a,d}{\asymp}} 2^{-\overline{s}j};
\end{align}
\item for any $x\in \Omega _{{\cal D}_{j,i},F}$
\begin{align}
\label{dist_x1} {\rm dist}_{|\cdot|}(x, \, \Gamma)
\underset{a,d}{\asymp} 2^{-\overline{s}j}.
\end{align}
\end{enumerate}

For $j\ge j_*$, $t\in \tilde I_j$, $l\in \Z_+$, we define the set
$\tilde I^l_{j,t}$ by
\begin{align}
\label{ijtl} \tilde I_{j,t}^l=\{i\in \tilde I_{j+l}:\, \hat
\xi_{j+l,i}\ge \hat \xi_{j,t}\}.
\end{align}

Let \label{a_def} ${\cal A}$ be the tree with vertex set
$\{\eta_{j,i}\} _{j\ge j_*,\, i\in \tilde I_j}$ and edge set
defined by
\begin{align}
\label{vaeij} {\bf V}^{\cal A}_1(\eta_{j,i}) =
\{\eta_{j+1,s}\}_{s\in \tilde I^1_{j,i}}.
\end{align}
For $\xi=\eta_{j,i}\in {\bf V}({\cal A})$ we write $\Omega[\xi]:=
\Omega_{j,i}:=\Omega_{{\cal D}_{j,i},F}$. We set $\hat\Theta
=\{\Omega[\xi]\} _{\xi\in {\bf V}({\cal A})}$, $\hat
F(\xi)=\Omega[\xi]$.

In \cite{vas_vl_raspr2} it is proved that
\begin{align}
\label{v_a_l_eji} {\bf V}^{\cal A}_l(\eta_{j,i})=\{\eta_{j+l,t}:\;
t\in \tilde I^l_{j,i}\}, \quad {\rm card}\, {\bf V}^{\cal
A}_l(\eta_{j,i}) \le K_0\frac{h(2^{-\overline{s}j})}
{h(2^{-\overline{s}(j+l)})}, \quad j, \; l\in \Z_+,
\end{align}
where $K_0=K_0({\mathfrak{Z}})$. In particular, (\ref{c_v1_a})
holds with $c_1=c_1({\mathfrak{Z}})$ and
\begin{align}
\label{sob_nu_t} {\rm card}\, \tilde I_j\underset{\mathfrak{Z}}
{\lesssim} \frac{2^{\overline{s}\theta j}}{\Lambda(2^{-\overline{s}j})}.
\end{align}

Denote by ${\cal P}_{r-1}(\R^d)$ the space of polynomials on
$\R^d$ of degree not exceeding $r-1$. For a measurable set
$E\subset \R^d$ we put $${\cal P}_{r-1}(E)= \{f|_E:\, f\in {\cal
P}_{r-1}(\R^d)\}.$$

\begin{trma}
\label{appr} {\rm (see \cite{vas_vl_raspr}).} Suppose that the
conditions (\ref{def_h}), (\ref{yty}), (\ref{ghi_g0}),
(\ref{psi_cond}), (\ref{muck}), (\ref{beta}), (\ref{phi_g}),
(\ref{ll}), (\ref{g0ag}) hold. Then $W^r_{p,g}(\Omega)\subset
L_{q,v}(\Omega)$. Moreover, for any subtree ${\cal D}\subset {\cal
T}_{\hat \xi_{j_0,i_0}}$ rooted at $\hat \xi_{j_0,i_0}$ there
exists a linear continuous operator $P:L_{q,v}(\Omega)\rightarrow
{\cal P}_{r-1}(\Omega)$ such that for any function $f\in
W^r_{p,g}(\Omega)$
\begin{align}
\label{pol_appr} \|f-Pf\|_{L_{q,v}(\Omega_{{\cal D},F})}
\underset{\mathfrak{Z}}{\lesssim} 2^{\overline{s}j_0\left(\beta
-\delta\right)}\Psi(2^{-\overline{s}j_0})\left\|\frac{\nabla^r
f}{g}\right\| _{L_p(\Omega_{{\cal D},F})}
\end{align}
in the case (\ref{beta}), a),
\begin{align}
\label{pol_appr1} \|f-Pf\|_{L_{q,v}(\Omega_{{\cal D},F})}
\underset{\mathfrak{Z}}{\lesssim} 2^{-\overline{s}\theta\left(\frac
1q-\frac 1p\right)_+j_0}j_0^{-\alpha+\left(\frac 1q-\frac
1p\right)_+}\rho(j_0)\left\|\frac{\nabla^r f}{g}\right\|
_{L_p(\Omega_{{\cal D},F})}
\end{align}
in the case (\ref{beta}), b).
\end{trma}
In order to prove this theorem the following two-weighted
Hardy-type inequality on the tree was obtained.
\begin{trma}
\label{two_weight_hardy_tree} Let $({\cal A}, \, \xi_0)$ be a tree
with vertex set $\{\eta_{j,i}\} _{j\ge j_*,\, i\in \tilde I_j}$
such that (\ref{vaeij}) and (\ref{v_a_l_eji}) hold, with $h\in
\mathbb{H}$ satisfying (\ref{def_h}), (\ref{yty}). Let $m_*\in
\N$, $\kappa_w>0$, $\kappa_u\in \R$, $\kappa:=\kappa_u+\kappa_w\ge
-\theta \left(\frac 1q-\frac 1p\right)_+$. Let $\Psi_g$, $\Psi_v$
be absolutely continuous functions, which satisfy
(\ref{psi_cond}). In addition, suppose that if
$\kappa=-\theta\left(\frac 1q-\frac 1p\right)_+$, then
(\ref{phi_g}), (\ref{ll}), (\ref{g0ag}) hold.  Let $u, \, w:{\bf
V}({\cal A})\rightarrow (0, \, \infty)$, $u(\xi)=u_j$,
$w(\xi)=w_j$, $\xi\in {\bf V}_{j-j_*}^{\cal A} (\xi_0)$,
$$
u_j=2^{-\kappa _um_*j}\Psi_g(2^{-m_*j}), \quad w_j=2^{-\kappa
_wm_*j}\Psi_v(2^{-m_*j}).
$$
Let ${\cal D}\subset {\cal A}$ with minimal vertex
$\xi_*=\eta_{j_0, i_0}$. Denote $\Psi(t)=\Psi_g(t)\Psi_v(t)$. Then
for any function $f:{\bf V}({\cal D}) \rightarrow \R_+$
$$
\left(\sum \limits _{\xi \in {\bf V}({\cal D})} w^q(\xi)
\left(\sum \limits _{\xi_*\le \xi'\le \xi} u(\xi')
f(\xi')\right)^q\right)^{\frac 1q}
\underset{p,q,u,w,h,K_0}{\lesssim} C(j_0)\left(\sum \limits _{\xi
\in {\bf V}({\cal D})} f^p(\xi) \right)^{\frac 1p},
$$
with $C(j_0)=2^{-\kappa m_*j_0}\Psi(2^{-m_*j_0})$ in the case
$\kappa>-\theta \left(\frac 1q-\frac 1p\right)_+$,
and $$C(j_0)=2^{-m_*\theta\left(\frac
1q-\frac 1p\right)_+j_0}j_0^{-\alpha+\left(\frac 1q-\frac
1p\right)_+}\rho(j_0)$$ in the case $\kappa=-\theta
\left(\frac 1q-\frac 1p\right)_+$.
\end{trma}

The following assertion (in fact, a more general result for
weighted spaces) was proved in \cite{vas_john}.  For the
non-weighted case, Besov \cite{besov_peak_width} later put forward a
more simple proof.
\begin{Lem}
\label{m_part} Let $G\subset \R^d$, $G\in {\bf FC}(a)$, $n\in
\N$. Then there exists a family of partitions $\{T_{m,n}(G)\}_{m\in
\Z_+}$ with the following properties:
\begin{enumerate}
\item ${\rm card}\, T_{m,n}(G)\underset{a,d}{\lesssim}2^mn$;
\item for any $E\in T_{m,n}(G)$ there exists a linear continuous
operator $P_E: L_q(\Omega) \rightarrow {\cal P}_r(E)$ such that for any
function $f\in {\rm span}\, W^r_p(G)$
\begin{align}
\label{lq_e} \|f-P_Ef\|_{L_q(E)}\underset{p,q,r,d,a}{\lesssim}
(2^mn)^{-\frac rd-\frac 1q+\frac 1p} ({\rm mes}\, G)^{\frac
rd+\frac 1q-\frac 1p} \|\nabla^r f\|_{L_p(E)};
\end{align}
\item for any $m\in \Z_+$, $E\in T_{m,n}(G)$,
$$
{\rm card}\, \{E'\in T_{m+1,n}(G):\; {\rm mes}\, (E\cap
E')>0\}\underset{a,d}{\lesssim}1,
$$
$$
{\rm card}\, \{E'\in T_{m-1,n}(G):\; {\rm mes}\, (E\cap
E')>0\}\underset{a,d}{\lesssim}1, \text{ if }m\ge 1.
$$
\end{enumerate}
\end{Lem}

\renewcommand{\proofname}{\bf Proof of Theorem \ref{main_sobol}}
\begin{proof}
{\it The upper estimate.} We apply Theorem \ref{main_abstr_th}. As
$X_p(\Omega)$ we take the linear span of $W^r_{p,g}(\Omega)$, as
$Y_q(\Omega)$ we take $L_{q,v}(\Omega)$, and as ${\cal P}(\Omega)$
we take the space of polynomials ${\cal P}_{r-1}(\Omega)$.

From Theorem \ref{appr} it follows that Assumption
\ref{sup1} holds with
\begin{align}
\label{w_pol} w_*(\eta_{j,i})=C(\mathfrak{Z})\cdot
2^{\overline{s}j\left(\beta
-\delta\right)}\Psi(2^{-\overline{s}j})
\end{align}
in the case (\ref{beta}), a),
\begin{align}
\label{w_pol1} w_*(\eta_{j,i}) =C(\mathfrak{Z})\cdot
2^{-\overline{s}\theta\left(\frac 1q-\frac 1p\right)_+j}
j^{-\alpha+\left(\frac 1q-\frac 1p\right)_+}\rho(j)
\end{align}
in the case (\ref{beta}), b).

By (\ref{ghi_g0}), (\ref{psi_cond}), Lemma \ref{sum_lem} and
(\ref{dist_x1}),
\begin{align}
\label{g_v} g(x)\underset{\mathfrak{Z}}{\asymp}
2^{\overline{s}\beta_gj}\Psi_g(2^{-\overline{s}j}), \quad
v(x)\underset{\mathfrak{Z}}{\asymp} 2^{\overline{s}\beta_vj}
\Psi_v(2^{-\overline{s}j}), \quad x\in \Omega_{j,i}.
\end{align}
From (\ref{diam_dist}), (\ref{diam_dj}) and the condition $\Omega_{j,i}\in
{\bf FC}(b_*)$ we get
\begin{align}
\label{mes_om_j} {\rm mes}\, \Omega_{j,i} \underset{a,d}{\asymp}
2^{-\overline{s}dj}.
\end{align}
By Lemma \ref{m_part}, (\ref{g_v}) and (\ref{mes_om_j}),
Assumption \ref{sup2} holds with $\delta_*=\frac{\delta}{d}$ and
$\tilde w_*(\eta_{j,i})=w_*(\eta_{j,i})$ in the case (\ref{beta}),
a), and
\begin{align}
\label{twji} \tilde w_*(\eta_{j,i})= C(\mathfrak{Z})\cdot
2^{-\overline{s}\theta\left(\frac 1q-\frac 1p\right)_+j}
j^{-\alpha}\rho(j)
\end{align}
in the case (\ref{beta}), b). Here $c_2=c_2(\mathfrak{Z})$.

Let us prove that Assumption \ref{sup3} holds. Then we apply
Theorem \ref{main_abstr_th} and obtain upper estimates of widths.

{\it Case 1.} Suppose that $\theta>0$ and the condition
(\ref{beta}), a) holds. Set $k_*=\overline{s}$, ${\cal
A}_{t,i}=\{\eta_{t,i}\}$. By (\ref{w_pol}), the relations
(\ref{w_s_2}) and (\ref{til_w_s_2}) hold with
$\lambda_*=\mu_*=\delta-\beta$, $u_*(y)=\Psi(y^{-1})$. The
relation (\ref{nu_t_k}) holds with $\gamma_*=\theta$,
$\psi_*(y)=\psi_\Lambda(y)=\frac{1}{\Lambda
\left(\frac{1}{y}\right)}$ (see (\ref{sob_nu_t})). Then
$\beta_*=\frac{1}{\theta}$. The function $\varphi_*=
\varphi_{\theta,\psi_\Lambda}$ is defined according to Lemma
\ref{obr} (see the notation on the page \pageref{phi_lam}).
Finally, (\ref{bipf4684gn}), (\ref{2l}) and (\ref{2ll}) hold. Here
$c_3=c_3(\mathfrak{Z})$. This implies the desired estimates.

{\it Case 2.} Suppose that $\theta>0$ and the condition
(\ref{beta}), b) holds. Set $k_*=\overline{s}$, ${\cal
A}_{t,i}=\{\eta_{t,i}\}$. From (\ref{phi_g}) and (\ref{sob_nu_t})
it follows that $\gamma_*=\theta$,
$$
\psi_*(t)=\frac{1}{\Lambda \left(\frac 1t\right)}=|\log
t|^{-\gamma} \tau^{-1}(|\log t|).
$$
Then $\beta_*=\frac{1}{\theta}$. By Lemma \ref{log}, there exists
$t_*(\mathfrak{Z})>1$ such that
\begin{align}
\label{ph_t} \varphi_*(t)\underset{\mathfrak{Z}} {\asymp}(\log
t)^{\frac{\gamma}{\theta}} \tau^{\frac{1} {\theta}}(\log t), \quad
t\ge t_*(\mathfrak{Z}).
\end{align}

Let $p>q$. From (\ref{w_pol1}) it follows that
$\lambda_*=\mu_*=\theta\left(\frac 1q-\frac 1p\right)$,
$\lambda_*\beta_*=\frac 1q-\frac 1p<\frac{\delta}{d}$. Therefore,
by assertion 1 of Theorem \ref{main_abstr_th}
$$\vartheta_n (W^r_{p,g}(\Omega), \, L_{q,v}(\Omega))
\stackrel{(\ref{vrth_n_pgq})}{\underset{\mathfrak{Z}}{\lesssim}}
\sigma_*(n)=u_*(n^{\beta_*}\varphi_*(n)) \varphi_*^{-\mu_*}(n).$$
From (\ref{w_pol1}) it follows that we can take $u_*(t)=(\log
t)^{-\alpha+\frac 1q-\frac 1p}\rho(\log t)$. Hence,
$$u_*(n^{\beta_*}\varphi_*(n)) \underset{\mathfrak{Z}}{\asymp} u_*(n)=(\log
n)^{-\alpha+\frac 1q-\frac 1p}\rho(\log n),$$ $$\varphi_*
^{-\mu_*}(n)\stackrel{(\ref{ph_t})}{\underset{\mathfrak{Z}}{\asymp}}
(\log n)^{-\gamma\left(\frac 1q-\frac 1p\right)} \tau ^{-\frac
1q+\frac 1p} (\log n).$$ Finally, (\ref{bipf4684gn}), (\ref{2l})
and (\ref{2ll}) hold. This implies the desired estimate.

Let $p\le q$. Then $\lambda_*=0$, $u_*(t)=(\log t)^{-\alpha}
\rho(\log t)$; in particular, (\ref{2l}) holds. Apply assertion 1
of Theorem \ref{main_abstr_th}. Since
$\lambda_*\beta_*=0<\frac{\delta}{d}$,
$$
\vartheta_n (W^r_{p,g}(\Omega), \, L_{q,v}(\Omega))
\stackrel{(\ref{vrth_n_pleq})}{\underset{\mathfrak{Z}}{\lesssim}}
u_*(n^{\beta_*} \varphi_*(n)) \varphi_*^{-\lambda_*}(n)
\underset{\mathfrak{Z}}{\lesssim} (\log n)^{-\alpha} \rho(\log n).
$$

{\it Case 3.} Suppose that $\theta=0$ and the condition (\ref{beta}),
b) holds. Then
\begin{align}
\label{ctij} {\rm card}\, \tilde I_j \stackrel{(\ref{phi_g}),
(\ref{sob_nu_t})}{\underset {\mathfrak{Z}} {\lesssim}}
(\overline{s}j)^{-\gamma}\tau^{-1}(\overline{s}j).
\end{align}
We set $k_*=1$ and define $\Gamma_t$ by
$${\bf V}(\Gamma_t)=\{\eta_{j,i}:\; j\ge j_*, \;\; 2^t\le j<2^{t+1}, \; i\in
\tilde I_j\}.$$ Then $\nu_t\stackrel{(\ref{ctij})}{\underset
{\mathfrak{Z}}{\lesssim}} 2^{(1-\gamma)t}\tau^{-1}(2^t)=:
\overline{\nu}_t$. Hence, (\ref{nu_t_k}) holds with
$\gamma_*=1-\gamma$, $\psi_*(x)=\tau^{-1}(x)$; therefore,
$\beta_*=\frac{1}{1-\gamma}$. Further,
$\lambda_*=\alpha-\left(\frac 1q-\frac 1p\right)_+$,
$u_*(x)=\rho(x)$ by (\ref{w_pol1}) and the condition $\theta=0$.

Let $p>q$. From (\ref{twji}) it follows that (\ref{til_w_s_2})
holds with $\mu_*=\alpha$. From the definition of $\tilde I_j$,
$\hat J_t$ and $\Gamma_t$, we have $\hat J_t=\tilde I_{2^t}$. By
(\ref{ctij}), ${\rm card}\, \hat J_t \underset
{\mathfrak{Z}}{\lesssim} 2^{-\gamma t}\tau^{-1}(2^t)$. Therefore,
(\ref{bipf4684gn}) holds.

Let $p\le q$. Then $\lambda_*=\mu_*=\alpha$,
$\lambda_*\beta_*=\frac{\alpha}{1-\gamma}$.

Notice that (\ref{2l}) and (\ref{2ll}) hold. It remains to apply
Theorem \ref{main_abstr_th}.

{\it Case 4}. Suppose that $\theta=0$ and the condition
(\ref{beta}), a) holds. Set $\tilde g(x)=\varphi_{\tilde g}({\rm
dist}_{|\cdot|}(x, \, \Gamma))$, $\varphi_{\tilde
g(t)}=t^{-\beta_{\tilde g}} |\log t|^{-\alpha_{\tilde g}}$,
$\beta_{\tilde g}+\beta_v=\delta$. Then $W^r_{p,g}(\Omega)\subset
W^r_{p,\tilde g}(\Omega)$ and $\left\| \frac{\nabla^r f}{\tilde
g}\right\|_{L_p(\Omega)} \underset{r,d,p, g,\tilde g}{\lesssim}
\left\| \frac{\nabla^r f}{g} \right\|_{L_p(\Omega)}$. It remains
to choose $\alpha_{\tilde g}$ so that the inequality
$\frac{\alpha_{\tilde g}+\alpha_v}{1-\gamma}> \frac{\delta} {d}$
holds (which implies (\ref{g0ag})), and to apply the result
obtained in the case 3.

{\it The lower estimate.} Recall that $R={\rm diam}\, \Omega$ (see
page \pageref{r_def}). Let $\xi_*$ be the minimal vertex of ${\cal
T}$, and let $l_*$ be the length of the side of the cube
$F(\xi_*)$. From the formulas (4.5), (4.6), (4.8) and (4.9) in
\cite{vas_vl_raspr} it follows that $l_*\underset{\mathfrak{Z}_*}
{\asymp} 1$ and for any $x\in F(\xi_*)$ the order equalities
$g(x)\underset {\mathfrak{Z}_*}{\asymp} 1$ and
$v(x)\underset{\mathfrak{Z}_*} {\asymp} 1$ hold. Hence,
\begin{align}
\label{vrt_n_gtrs_1cube}
\vartheta_n(W^r_{p,g}(\Omega), \, L_{q,v}(\Omega)) \underset{\mathfrak{Z}_*}
{\gtrsim} \vartheta _n(W^r_p([0, \, 1]^d), \, L_q([0, \, 1]^d)).
\end{align}
We apply Theorem \ref{sob_dn} and obtain the desired estimate in
the case 4.

In cases 1, 2, 3 we apply the construction given at the end of the
paper \cite{vas_vl_raspr2}. For $t\in \N$, $t\ge
t_0(\mathfrak{Z}_*)$, let the index sets $J_t$ and the functions
$\psi _{t,j} \in W^r_{p,g}(\Omega)$ ($j\in J_t$) be as defined in
\cite{vas_vl_raspr2}. These functions satisfy the following
conditions: $\left\|\frac{\nabla^r \psi_{t,j}} {g}\right\|
_{L_p(\Omega)}=1$,
\begin{align}
\label{norm_psi_tj} \|\psi_{t,j}\|_{L_{q,v}(\Omega)}
\underset{\mathfrak{Z}}{\asymp}
2^{k_{**}t(\beta-\delta)}\Psi(2^{-k_{**}t}),
\end{align}
${\rm supp}\, \psi_{t,j}\subset \Delta _{t,j}$, the cubes
$\Delta_{t,j}$ do not overlap pairwise for different $(t, \, j)$,
$k_{**}\in \N$, $k_{**} \underset{\mathfrak{Z}}{\asymp} 1$,
\begin{align}
\label{card_j_t} {\rm card}\, J_t \underset{\mathfrak{Z}}{\asymp}
\frac{2^{k_{**}t\theta}}{\Lambda (2^{-k_{**}t})}.
\end{align}

Consider cases 1 and 2. Given $\nu\in \N$, we set
\begin{align}
\label{t_st} t_*=t_*(\nu)=\lfloor
k_{**}^{-1}\log(\nu^{1/\theta}\varphi
_{\theta,\psi_\Lambda}(\nu))\rfloor.
\end{align}
From Lemma \ref{sum_lem} it follows that ${\rm card}\, J_{t_*}
\stackrel{(\ref{card_j_t})}{\underset{\mathfrak{Z}}{\asymp}}
(\nu^{1/\theta}\varphi_{\theta,\psi_\Lambda}(\nu))^\theta
\psi_\Lambda(\nu^{1/\theta}\varphi_{\theta,\psi_\Lambda}(\nu))=
\nu$. We choose $\nu=\nu(n)$ so that
\begin{align}
\label{2n_jt} 2n\le {\rm card}\, J_{t_*}
\underset{\mathfrak{Z}}{\lesssim} n.
\end{align}
Then $\nu\underset{\mathfrak{Z}}{\asymp} n$. From Lemma
\ref{low_est} it follows that
\begin{align}
\label{vrth_n_theta3} \vartheta _n(W^r_{p,g}(\Omega), \,
L_{q,v}(\Omega)) \stackrel{(\ref{norm_psi_tj})} {\underset
{\mathfrak{Z}} {\gtrsim}} (n^{1/\theta}\varphi_{\theta,
\psi_\Lambda} (n))^{\beta-\delta} \Psi(n^{-\frac{1}{\theta}}
\varphi^{-1}_{\theta,\psi_\Lambda}(n)) \cdot \vartheta_n(B^{2n}_p,
\, l_q^{2n}).
\end{align}
For $p<q$, $\hat q>2$ we also take
$t_{**}=[k_{**}^{-1}\log(\nu^{\frac {\hat q}{2\theta}}\varphi
_{\theta,\psi_\Lambda}(\nu^{\hat q/2}))]$. Then ${\rm card}\,
J_{t_{**}} \stackrel{(\ref{card_j_t})} {\underset{\mathfrak{Z}}
{\asymp}} \nu^{\hat q/2}$. We choose $\nu=\nu(n)$ so that
$$
2n^{\hat q/2}\le {\rm card}\, J_{t_{**}} \underset{\mathfrak{Z}}
{\lesssim} n^{\hat q/2}.
$$
We apply Lemma \ref{low_est} together with Theorem \ref{glus_trm}
and get
\begin{align}
\label{vrth_n_theta4} \vartheta _n(W^r_{p,g}(\Omega), \,
L_{q,v}(\Omega)) \stackrel{(\ref{norm_psi_tj})}
{\underset{\mathfrak{Z}} {\gtrsim}} (n^{\hat
q/2\theta}\varphi_{\theta,\psi_\Lambda}(n^{\hat
q/2}))^{\beta-\delta} \Psi(n^{-\frac{\hat
q}{2\theta}}\varphi^{-1}_{\theta,\psi_\Lambda}(n^{\hat q/2})).
\end{align}
Consider case 1. From (\ref{vrt_n_gtrs_1cube}),
(\ref{vrth_n_theta3}), (\ref{vrth_n_theta4}),
(\ref{width_pietsch_stesin}) and Theorems \ref{glus_trm} and
\ref{sob_dn} we obtain the desired estimate.

Consider case 2. Let $p\le q$. If $p=q$ or $p<q$, $\hat q\le 2$,
we have $\vartheta _n(B_p^{2n}, \, l_q^{2n}) \underset{p,q}
{\asymp} 1$. Since $\beta=\delta$, we get
$$
\vartheta _n(W^r_{p,g}(\Omega), \, L_{q,v}(\Omega))
\stackrel{(\ref{phi_g}),(\ref{vrth_n_theta3})}{\underset{\mathfrak{Z}}
{\gtrsim}}
$$
$$
\gtrsim\log ^{-\alpha}(n^{1/\theta}
\varphi_{\theta,\psi_\Lambda}(n)) \rho(\log [n
^{1/\theta}\varphi_{\theta,\psi_\Lambda}(n)])
\underset{\mathfrak{Z}}{\asymp} (\log n)^{-\alpha} \rho(\log n)
$$
(see Lemma \ref{sum_lem}). If $p<q$ and $\hat q>2$, we get
$$
\vartheta _n(W^r_{p,g}(\Omega), \, L_{q,v}(\Omega))
\stackrel{(\ref{phi_g}),(\ref{vrth_n_theta4})}{\underset{\mathfrak{Z}}
{\gtrsim}}
$$
$$
\gtrsim \log ^{-\alpha}(n^{\frac{\hat q}{2\theta}}
\varphi_{\theta,\psi_\Lambda}(n^{\hat q/2})) \rho(\log
[n^{\frac{\hat q}{2\theta}} \varphi _{\theta,\psi_\Lambda}(n^{\hat
q/2})]) \underset{\mathfrak{Z}}{\asymp} (\log n)^{-\alpha}
\rho(\log n).
$$

Let $p>q$. Take $t_*$ such that (\ref{t_st}) and
(\ref{2n_jt}) hold. Then, by Lemma \ref{sum_lem},
$t_*\underset{\mathfrak{Z}}{\asymp} \log n$. Since ${\rm card}\,
J_{t_*}>n$, Lemma \ref{low_est} implies that
$$
\vartheta _n(W^r_{p,g}(\Omega), \, L_{q,v}(\Omega))
\stackrel{(\ref{tn_low_est_diag}),(\ref{norm_psi_tj})}{\underset{\mathfrak{Z}}
{\gtrsim}} \left(\sum \limits _{t=t_*+1}^{2t_*} \sum \limits
_{j\in J_t} \left(2^{k_{**}t(\beta-\delta)}
\Psi(2^{-k_{**}t})\right) ^{\frac{pq}{p-q}}\right)^{\frac 1q-\frac
1p} \stackrel{(\ref{card_j_t})}{\underset{\mathfrak{Z}} {\gtrsim}}
$$
$$
\gtrsim \left(\sum \limits _{t=t_*+1}^{2t_*}
\left(2^{k_{**}t(\beta-\delta)} \Psi(2^{-k_{**}t})
2^{k_{**}t\theta\left(\frac 1q-\frac 1p\right)} \Lambda ^{\frac
1p-\frac 1q}(2^{k_{**}t})\right)^{\frac{pq}{p-q}}\right)^{\frac
1q-\frac 1p} \stackrel{(\ref{phi_g})} {\underset {\mathfrak{Z}}
{\asymp}}
$$
$$
\asymp\left(\sum \limits _{t=t_*+1}^{2t_*}
\left(t^{-\alpha}\rho(t) t^{\gamma\left(\frac 1p-\frac 1q\right)}
\tau ^{\frac 1p-\frac 1q}(t)\right)
^{\frac{pq}{p-q}}\right)^{\frac 1q-\frac 1p}
\underset{\mathfrak{Z}}{\asymp}
$$
$$
\asymp t_*^{(1-\gamma)\left(\frac 1q-\frac 1p\right)-\alpha} \rho(t_*)
\tau ^{\frac 1p-\frac 1q}(t_*) \underset{\mathfrak{Z}}{\asymp}
(\log n)^{(1-\gamma)\left(\frac 1q-\frac 1p\right)-\alpha} \rho(\log n)
\tau ^{\frac 1p-\frac 1q}(\log n).
$$

Consider case 3. Let $p\le q$. Then
$$
\|\psi _{m,j}\| _{L_{q,v}
(\Omega)} \stackrel{(\ref{phi_g}),(\ref{norm_psi_tj})}
{\underset{\mathfrak{Z}}{\asymp}} m^{-\alpha} \rho(m).
$$
For $2^t\le m\le 2^{t+1}$ we get
\begin{align}
\label{psi_mj} \|\psi _{m,j}\| _{L_{q,v} (\Omega)}
\underset{\mathfrak{Z}} {\asymp} 2^{-t\alpha} \rho(2^t),
\end{align}
\begin{align}
\label{cjm} {\rm card}\, J_m
\stackrel{(\ref{phi_g}),(\ref{card_j_t})}{\underset
{\mathfrak{Z}}{\asymp}} m^{-\gamma} \tau^{-1}(m) \underset
{\mathfrak{Z}}{\asymp} 2^{-\gamma t}(\tau(2^t))^{-1},
\end{align}
\begin{align}
\label{m2t_jm} \sum \limits _{m=2^t}^{2^{t+1}} {\rm card}\, J_m
\underset{\mathfrak{Z}}{\asymp}
 2^{(1-\gamma)t}(\tau(2^t))^{-1}.
\end{align}
For $\nu\in \N$ we set
\begin{align}
\label{t_st1} t_*=\left[\log \left(\nu^{\frac{1}{1-\gamma}}
\varphi_{1-\gamma,\tau^{-1}}(\nu)\right)\right].
\end{align}
Then
\begin{align}
\label{21mg} 2^{(1-\gamma)t_*}(\tau(2^{t_*}))^{-1}
\underset{\mathfrak{Z}}{\asymp} \nu.
\end{align}
We choose $\nu=\nu(n)$ so that
\begin{align}
\label{2n_sum} 2n\le \sum \limits _{m=2^{t*}}^{2^{t_*+1}} {\rm
card}\, J_m \underset{\mathfrak{Z}}{\lesssim} n, \quad \sum
\limits _{m=2^{t*-1}}^{2^{t_*}} {\rm card}\, J_m>n
\end{align}
(then $\nu \stackrel{(\ref{m2t_jm}), (\ref{21mg})}
{\underset{\mathfrak{Z}}{\asymp}} n$). Apply Lemma \ref{low_est}.
By (\ref{psi_mj}) and (\ref{t_st1}), we get
\begin{align}
\label{case4_low_t_st}
\vartheta _n(W^r_{p,g}(\Omega), \, L_{q,v}(\Omega)) \underset{\mathfrak{Z}}
{\gtrsim} n^{-\frac{\alpha}{1-\gamma}} \varphi^{-\alpha}_{1-\gamma,\tau^{-1}}
(n)\rho\left(n^{\frac{1}{1-\gamma}}\varphi_{1-\gamma,\tau^{-1}}(n)\right)
\vartheta_n(B_p^{2n}, \, l_q^{2n}).
\end{align}
For $p<q$, $\hat q>2$ we also consider
$t_{**}=\left[\log \left(\nu^{\frac{\hat q}
{2(1-\gamma)}} \varphi_{1-\gamma,\tau^{-1}}
(\nu^{\frac{\hat q}{2}})\right)\right]$.
Then
$$
2^{(1-\gamma)t_{**}}(\tau(2^{t_{**}}))^{-1}
\underset{\mathfrak{Z}}{\asymp} \nu ^{\frac{\hat q}{2}}.
$$
We choose $\nu=\nu(n)$ so that
$$2n^{\frac{\hat q}{2}}\le \sum \limits _{m=2^{t_{**}}}^{2^{t_{**}+1}}
{\rm card}\, J_m \underset{\mathfrak{Z}}{\lesssim} n^{\frac{\hat
q}{2}}.$$ Then $\nu \stackrel{(\ref{m2t_jm})} {\underset
{\mathfrak{Z}} {\asymp}} n$. By Lemma \ref{low_est}, Theorem
\ref{glus_trm} and (\ref{psi_mj}),
$$
\vartheta _n(W^r_{p,g}(\Omega), \, L_{q,v}(\Omega)) \underset{\mathfrak{Z}}
{\gtrsim} n^{-\frac{\hat q\alpha}{2(1-\gamma)}} \varphi^{-\alpha}_{1-\gamma,\tau^{-1}}
(n^{\frac{\hat q}{2}})\rho\left(n^{\frac{\hat q}{2(1-\gamma)}}
\varphi_{1-\gamma,\tau^{-1}}(n^{\frac{\hat q}{2}})\right).
$$

This together with (\ref{vrt_n_gtrs_1cube}), Theorem
\ref{sob_dn} and (\ref{case4_low_t_st})
implies the desired estimate.

Let $p>q$. We take $t_*$ such that (\ref{t_st1}) and
(\ref{2n_sum}) hold. Then $\nu \underset{\mathfrak{Z}} {\asymp}
n$. Apply (\ref{tn_low_est_diag}) taking into account the second
inequality in (\ref{2n_sum}). We have
$$
\vartheta _n(W^r_{p,g}(\Omega), \, L_{q,v}(\Omega))
\stackrel{(\ref{psi_mj})} {\underset{\mathfrak{Z}} {\gtrsim}}
\left(\sum \limits _{m=2^{t_*}+1}^{2^{t_*+1}}\sum \limits _{j\in
J_m} \left(2^{-t_*\alpha}\rho(2^{t_*})\right)
^{\frac{pq}{p-q}}\right)^{\frac 1q-\frac 1p}
\stackrel{(\ref{cjm})}{\underset{\mathfrak{Z}}{\asymp}}
$$
$$
\asymp 2^{-t_*\alpha}\rho(2^{t_*})\left(2^{(1-\gamma)t_*}\tau^{-1}
(2^{t_*})\right)^{\frac 1q-\frac 1p}
\stackrel{(\ref{t_st1}),(\ref{21mg})}
{\underset{\mathfrak{Z}}{\asymp}}$$$$\asymp
n^{-\frac{\alpha}{1-\gamma}+\frac 1q-\frac 1p}
\rho(n^{\frac{1}{1-\gamma}} \varphi_{1-\gamma,\tau^{-1}}(n))
\varphi^{-\alpha}_{1-\gamma,\tau^{-1}}(n).
$$
This completes the proof.
\end{proof}
\renewcommand{\proofname}{\bf Proof}
\section{Estimates for widths of function classes
on a metric and combinatorial tree}

First we give necessary notations.

Let $(\cal{T}, \, \xi_*)$ be a tree, let $\Delta:{\bf
E}({\cal{T}})\rightarrow 2^{\R}$ be a mapping such that for any
$\lambda\in {\bf E}(\cal{T})$ the set $\Delta(\lambda)=[a_\lambda,
\, b_\lambda]$ is a non-degenerate segment. By a metric tree we
mean
$$
\mathbb{T}=({\cal T}, \, \Delta)=\{(t, \, \lambda):\, t\in
[a_\lambda, \, b_\lambda], \; \lambda\in {\bf E}({\cal{T}})\};
$$
here we suppose that if $\xi'\in {\bf V}_1(\xi)$, $\xi''\in {\bf
V}_1(\xi')$, $\lambda=(\xi, \, \xi')$, $\lambda'=(\xi', \,
\xi'')$, then $(b_\lambda, \, \lambda)=(a_{\lambda'}, \,
\lambda')$. The distance between points of $\mathbb{T}$ is defined
as follows: if $(\xi_0, \, \xi_1, \, \dots, \, \xi_n)$ is a simple
path in the tree $\mathcal{T}$, $n\ge 2$, $\lambda_i=(\xi_{i-1},
\, \xi_i)$, $x=(t_1, \, \lambda_1)$, $y=(t_n, \, \lambda_n)$, then
we set
$$
|y-x|_{\mathbb{T}}=|b_{\lambda_1}-t_1|+\sum \limits _{i=2}^{n-1}
|b_{\lambda_i}-a_{\lambda_i}| +|t_n-a_{\lambda_n}|;
$$
if $x=(t', \, \lambda)$, $y=(t'', \, \lambda)$, then we set
$|y-x|_{\mathbb{T}}=|t'-t''|$.

We say that $\lambda<\lambda'$ if $\lambda=(\omega, \, \xi)$,
$\lambda'=(\omega', \, \xi')$ and $\xi\le \omega'$; $(t', \,
\lambda')\le(t'', \, \lambda'')$ if either $\lambda'<\lambda''$ or
$\lambda'=\lambda''$, $t'\le t''$. If $(t', \, \lambda')\le (t'',
\, \lambda'')$ and $(t', \, \lambda') \ne (t'', \, \lambda'')$,
then we write $(t', \, \lambda')<(t'', \, \lambda'')$. For $a$,
$x\in \mathbb{T}$, $a\le x$, we set $[a, \, x]=\{y\in
\mathbb{T}:\, a\le y\le x\}$.

Let $A_\lambda\subset \Delta(\lambda)$, $\lambda\in {\bf
E}(\mathcal{T})$. We say that the subset $\mathbb{A}=\{(t, \,
\lambda):\; \lambda\in {\bf E}(\mathcal{T}), \; t\in A_\lambda\}$
is measurable if $A_\lambda$ is measurable for any $\lambda\in
{\bf E}(\mathcal{T})$. Its Lebesgue measure is defined by
$$
{\rm mes}\, \mathbb{A}=\sum \limits _{\lambda\in {\bf
E}(\mathcal{T})} {\rm mes}\, A_\lambda.
$$

Let $f:\mathbb{A}\rightarrow \R$. For $\lambda\in {\bf
E}(\mathcal{T})$ we define the function $f_\lambda:A_\lambda
\rightarrow \R$ by $f_\lambda(t)=f(t, \, \lambda)$. We say that
the function $f:\mathbb{A}\rightarrow \R$ is Lebesgue integrable
if $f_\lambda$ is Lebesgue integrable for any $\lambda\in {\bf
E}(\mathcal{T})$ and $\sum \limits _{\lambda \in {\bf
E}(\mathcal{T})} \int \limits _{A_\lambda} |f_\lambda(t)|\, dt<
\infty$, and write
$$
\int \limits _{\mathbb{A}} f(x)\, dx=\sum \limits _{\lambda \in
{\bf E}(\mathcal{T})} \int \limits _{A_\lambda} f_\lambda(t)\, dt.
$$

Let $\mathbb{T}=({\cal T}, \, \Delta)$ be a metric tree, $x_0\in
\mathbb{T}$, and let $g$, $v: \mathbb{T} \rightarrow (0, \,
\infty)$ be measurable functions. The two-weighted Hardy-type
operator $I_{g,v,x_0}$ on the metric tree $\mathbb{T}$ is defined
by
$$
I_{g,v,x_0}f(x)=v(x)\int \limits _{[x_0, \, x]} g(t)f(t)\, dt.
$$
The criterion of boundedness for such operator and order estimate
of its norm (up to some absolute constants) were obtained by
Evans, Harris and Pick \cite{ev_har_pick}.

Let us formulate the main result of this section.

Let $\mathbb{T}=({\cal T}, \, \Delta)$ be a metric tree, and let
$g$, $v: \mathbb{T} \rightarrow (0, \, \infty)$ be measurable
functions. We suppose that the following conditions hold.
\begin{enumerate}
\item Let a regular partition $\{({\cal T}_{k,i}, \,
\xi_{k,i})\}_{k\in \Z_+, \, i\in I_k}$ of the tree ${\cal T}$ and
a function $h\in \mathbb{H}$ be given, and let $h$ be as defined
by (\ref{def_h}) in some neighborhood of zero. Let $\tilde{\cal
T}_{k,i}$ be a tree with vertex set ${\bf V}(\tilde{\cal T}_{k,i})
={\bf V}({\cal T}_{k,i})\cup \{\xi_{k+1,i'}:\; \xi_{k+1,i'}>\xi
_{k,i}\}$.
\item There exists $c_*\ge 1$ such that for any $k$,
$l\in \Z_+$, $i\in I_k$
\begin{align}
\label{card_metr} {\rm card}\, \{i'\in I_{k+l}:\;
\xi_{k+l,i'}>\xi_{k,i}\}\le c_* \frac{h(2^{-k})} {h(2^{-k-l})}.
\end{align}
\item For any $k\in \Z_+$, $i\in I_k$
\begin{align}
\label{norm_g_metr} \|g\|_{L_{p'}(\mathbb{T}_{k,i})}\le c_*\cdot
2^{\left(\beta_g -\frac{1}{p'}\right)k}\Psi_g(2^{-k})=:c_*u_k,
\end{align}
\begin{align}
\label{norm_v_metr} \|v\|_{L_q(\mathbb{T}_{k,i})}\le c_*\cdot
2^{\left(\beta_v -\frac 1q\right)k}\Psi_v(2^{-k})=:c_*w_k,
\end{align}
where functions $\Psi_g$, $\Psi_v$ satisfy (\ref{psi_cond}) and
$\mathbb{T}_{k,i}=(\tilde{\cal T}_{k,i}, \, \Delta)$.
\item The relations (\ref{muck}) and
(\ref{beta}) with $d=1$ and $\delta=1+\frac 1q-\frac 1p$ hold;
moreover, in the case (\ref{beta}), b) the relations
(\ref{phi_g}), (\ref{ll}), (\ref{g0ag}) hold.
\end{enumerate}

Denote $\mathfrak{Z}=(p, \, q, \, h, \, g, \, v, \, c_*)$.

Let $x_0$ be the minimal point in $\mathbb{T}$. We set
$$W^1_{p,g}(\mathbb{T})=\{I_{g,1,x_0}f:\;
\|f\|_{L_p(\mathbb{T})}\le 1\}.$$

\begin{Trm}
\label{main_metr1} We set $d=1$, $\delta=1-\frac 1p+\frac 1q$.
\begin{enumerate}
\item Suppose that $\theta>0$ and the condition (\ref{beta}), a) holds.
Let $p\ge q$ or $\hat q\le 2$. Define $\theta_1$, $\theta_2$,
$\sigma_1$ and $\sigma_2$ by (\ref{case1_theta_j_0}) and
(\ref{case1_sigma_j_0}). Suppose that $\theta_1\ne \theta_2$,
$j_*\in \{1, \, 2\}$, $\theta_{j_*}=\min\{\theta_1, \,
\theta_2\}$. Then
\begin{align}
\label{vrth_n_def} \vartheta_n(W^1_{p,g}(\mathbb{T}), \,
L_{q,v}(\mathbb{T})) \underset{\mathfrak{Z}}{\lesssim}
n^{-\theta_{j_*}} \sigma_{j_*}(n).
\end{align}
Let $p<q$ and $\hat q>2$. Define $\theta_j$ and $\sigma_j$ ($j=1,
\, 3, \, 4$) by (\ref{case1qg2_th12}), (\ref{case1qg2_th34}),
(\ref{case1def_sigma_qg2}). Suppose that there exists $j_*\in \{1,
\, 3, \, 4\}$ such that (\ref{thj_min_case1}) holds. Then we have
(\ref{vrth_n_def}).
\item Suppose that $\theta>0$ and the condition (\ref{beta}), b) holds.
Then $$\vartheta_n(W^1_{p,g}(\mathbb{T}), \, L_{q,v}(\mathbb{T}))
\underset{\mathfrak{Z}}{\lesssim} (\log
n)^{-\alpha+(1-\gamma)\left(\frac 1q-\frac 1p\right)_+} \rho(\log
n)\tau^{-\left(\frac 1q-\frac 1p\right)_+}(\log n).$$
\item Suppose that $\theta=0$ and the condition (\ref{beta}), b) holds.
Let $p\ge q$ or $\hat q\le 2$, let $\theta_1$, $\theta_2$,
$\sigma_1$ and $\sigma_2$ be defined by (\ref{case3_pleq_theta_j})
and (\ref{case3_pleq_sigma}), and let $\theta_1\ne \theta_2$,
$j_*\in \{1, \, 2\}$, $\theta_{j_*}=\min\{\theta_1, \,
\theta_2\}$. Then (\ref{vrth_n_def}) holds. Let $p<q$ and $\hat
q>2$, let $\theta_j$ and $\sigma_j$ ($j=1, \, 3, \, 4$) be defined
by (\ref{case3_qg2_th12}), (\ref{case3_qg2_th34}),
(\ref{case3_qg2_sigma}). Suppose that there exists $j_*\in \{1, \,
3, \, 4\}$ such that (\ref{case3_qg2_ineq}) holds. Then we have
(\ref{vrth_n_def}).
\item Suppose that $\theta=0$ and (\ref{beta}), a) holds.
Then $\vartheta_n(W^1_{p,g}(\mathbb{T}), \, L_{q,v}(\mathbb{T}))
\underset{\mathfrak{Z}}{\lesssim} n^{-\theta_{p,q,1,1}}$, where
$\theta_{p,q,1,1}$ is defined by (\ref{sob_dn_theta_pqrd}) for
$r=1$, $d=1$.
\end{enumerate}
Suppose, in addition, that the following conditions hold:
\begin{itemize}
\item ${\rm card}\, I_l\ge \frac{1}{c_* h(2^{-l})}$ for any $l\in \Z_+$,
\item for any $k\in \Z_+$, $i\in I_k$ there exists a function
$\psi_{k,i}=I_{g,1,x_0}f_{k,i}$ such that ${\rm supp}\, \psi_{k,i}
\subset \mathbb{T}_{k,i}$ and $\|f_{k,i}\| _{L_p(\mathbb{T})}=1$,
$\|\psi_{k,i}\|_{L_{q,v}(\mathbb{T})} \ge c_*^{-1} u_kw_k$.
\end{itemize}
Then similar lower estimates are true for
$\vartheta_n(W^1_{p,g}(\mathbb{T}), \, L_{q,v}(\mathbb{T}))$.
\end{Trm}

We need the following result.
\begin{Lem}
\label{metr_razb} Let $\mathbb{A}=({\cal A}, \, \Delta)$ be a
metric tree. Suppose that there is $\hat k\in \N$ such that ${\rm
card}\, {\bf V}_1(\xi)\le \hat k$ for any $\xi\in {\bf V}({\cal
A})$. Let $\Phi$ be a function defined on the family of measurable
subsets of $\mathbb{A}$, satisfying the following conditions:
\begin{align}
\label{pr1} \Phi(A_1\sqcup A_2)\ge \Phi(A_1)+\Phi(A_1), \quad A_1,
\, A_2\subset \mathbb{A},
\end{align}
\begin{align}
\label{pr2} \Phi(A)\to 0\quad \text{as}\quad {\rm mes}\, A\to 0.
\end{align}
Then for any $n\in \N$, $m\in \Z_+$ there exists a partition ${\bf
P}_{m,n}$ of the tree $\mathbb{A}$ into subtrees with the
following properties:
\begin{enumerate}
\item ${\rm card}\, {\bf P}_{m,n}\underset{\hat k}{\lesssim} 2^mn$,
\item $\Phi(E)\underset{\hat k}{\lesssim} \frac{\Phi(\mathbb{A})}{2^mn}$
for any $E\in {\bf P}_{m,n}$,
\item ${\rm card}\, \{E'\in {\bf P}
_{m\pm 1, \, n}:\; {\rm mes}(E\cap E')>0\} \underset{\hat
k}{\lesssim} 1$ for any $E\in {\bf P}_{m,n}$.
\end{enumerate}
\end{Lem}
The partition $\mathbf{P}_{0,n}$ was constructed by Solomyak
\cite{solomyak}. In order to construct partitions
$\mathbf{P}_{m,n}$, we repeat arguments in \cite[Lemma
8]{vas_john}.

As an example of a function $\Phi$ satisfying (\ref{pr1}) and
(\ref{pr2}) we can take $\Phi(A)=\prod _{j=1}^l
(\mu_j(A))^{\alpha_j}$, where $\mu_j$ are absolutely continuous
measures, $\alpha_j>0$, $\sum \limits_{j=1}^l \alpha_j=1$. It
follows from H\"{o}lder's inequality and Radon -- Nikodym theorem
(see \cite{vas_john}).

\renewcommand{\proofname}{\bf Proof of Theorem \ref{main_metr1}}
\begin{proof}
Let ${\cal D}\subset {\cal T}$ be the tree such that ${\bf
V}({\cal D}) = \cup _{(j,\, i)\in \hat J} {\bf V}({\cal T}_{j,i})$
for some $\hat J$. By $\tilde{\cal D}$ we denote a tree with
vertex set $${\bf V}(\tilde{\cal D})={\bf V}({\cal D}) \cup
\{\xi'\in {\bf V}_1^{{\cal T}}(\xi):\; \xi \in {\bf
V}_{\max}({\cal D})\}.$$ Denote $\mathbb{D}=(\tilde{\cal D}, \,
\Delta)$. Let $\xi_{j_0, \, i_0}$ be the minimal vertex in ${\cal
D}$, and let $x_*$ be the minimal point in $\mathbb{D}$. We
estimate the norm of $\|I_{g,v,x_*}\|_{L_p(\mathbb{D}) \rightarrow
L_q(\mathbb{D})}$.

Let $\|f\|_{L_p(\mathbb{D})}=1$, $f\ge 0$. Denote
$f_{j,i}=\|f\|_{L_p(\mathbb{T}_{j,i})}$. Then
$$
\left(\int \limits _{\mathbb{D}} v^q(x)\left(\int \limits _{x_*}
^x g(t)f(t)\, dt \right)^q\, dx\right)^{\frac 1q}=
$$
$$
= \left(\sum \limits _{(j,\, i)\in \hat J\,} \int \limits
_{\, \mathbb{T}_{j,i}} v^q(x)\left(\sum \limits _{\xi_{j_0, \,
i_0}\le \xi _{j',i'}\le \xi_{j,i}\,} \int \limits _{\, [x_*, \,
x]\cap \mathbb{T}_{j',i'}} g(t)f(t)\, dt\right)^q\,
dx\right)^{\frac 1q} \stackrel{(\ref{norm_g_metr}),\,
(\ref{norm_v_metr})}{\underset{\mathfrak{Z}}{\lesssim}}
$$
$$
\lesssim \left(\sum \limits _{(j,\, i)\in \hat J} w_j^q \left(
\sum \limits _{\xi_{j_0, \, i_0}\le \xi _{j',i'}\le
\xi_{j,i}} u_{j'}f_{j',i'}\right)^q\right)^{\frac 1q}.
$$
We apply Theorem \ref{two_weight_hardy_tree} and get
\begin{align}
\label{igvx1} \|I_{g,v,x_*}\|_{L_p(\mathbb{D}) \rightarrow
L_q(\mathbb{D})} \underset{\mathfrak{Z}}{\lesssim}
2^{j_0(\beta-\delta)} \Psi(2^{-j_0})
\end{align}
in the case (\ref{beta}), a),
$$
\|I_{g,v,x_*}\|_{L_p(\mathbb{D})
\rightarrow L_q(\mathbb{D})} \underset{\mathfrak{Z}}{\lesssim}
2^{-\theta\left(\frac 1q-\frac
1p\right)_+j_0}j_0^{-\alpha+\left(\frac 1q-\frac
1p\right)_+}\rho(j_0)
$$
in the case (\ref{beta}), b).  Moreover, we have in the case
(\ref{beta}), b)
\begin{align}
\label{igvx2} \|g\|_{L_{p'}(\mathbb{T}_{j,i})}
\|v\|_{L_q(\mathbb{T}_{j,i})} \underset{\mathfrak{Z}}{\lesssim}
2^{-\theta\left(\frac 1q-\frac 1p\right)_+j}j^{-\alpha}\rho(j).
\end{align}
Let
$$
\Phi(A)=\left(\int \limits_A g^{p'}(x)\,
dx\right)^{\frac{1}{p'\left(1-\frac 1p+\frac 1q\right)}}
\left(\int \limits _A v^q(x)\, dx\right)^{\frac{1}{q\left(1-\frac
1p+\frac 1q\right)}}.
$$
Then (\ref{card_metr}), Lemma \ref{metr_razb}, (\ref{igvx1}) and
(\ref{igvx2}) yield Assumptions 1, 2 and 3. Repeating arguments of
Theorem \ref{main_sobol} and taking into account that for $\hat
q>2$ the inequality $\frac{\hat q\delta}{2}>\delta+\frac
12-\frac{1}{\hat q}$ holds, we obtain the desired estimates for
widths.
\end{proof}
\renewcommand{\proofname}{\bf Proof}

Consider the two-weighted summation operator on a combinatorial
tree ${\cal A}$. Let ${\bf V}({\cal A}) =\{\eta_{j,i}:\; j\in
\Z_+, \;\; i\in I_j\}$. Suppose that for some $c_*\ge 1$, $m_*\in
\N$
$$
{\bf V}^{\cal A}_l(\eta_{j,i})\le  c_*\frac{h(2^{-m_*j})}
{h(2^{-m_*(j+l)})}, \quad j, \; l\in \Z_+,
$$
where the function $h$ has the form (\ref{def_h}) in some
neighborhood of zero. Let the functions $u$, $w:{\bf V}({\cal A})
\rightarrow (0, \, \infty)$ be such as in Theorem
\ref{two_weight_hardy_tree}.

We set
$$
S_{u,w}:l_p({\cal A})\rightarrow l_q({\cal A}), \quad
S_{u,w}f(\xi)=w(\xi) \sum \limits _{\xi'\le \xi} u(\xi')f(\xi'),
$$
where $f\in l_p({\cal A})$ (see (\ref{flpg})); denote
$$
{\bf S}_{p,u,w}({\cal A})=\{S_{u,w}f:\; \|f\|_{l_p({\cal A})}\le
1\}.
$$
Write $\mathfrak{Z}=(p, \, q, \, u, \, w, \, h)$.

Applying Theorem \ref{two_weight_hardy_tree}, Remark
\ref{diskr_case} and arguing as in Theorem \ref{main_sobol} in
obtaining lower estimates, we get the following result.
\begin{Trm}
\begin{enumerate}
\item Suppose that $\theta>0$ and $\kappa >-\theta\left(\frac 1q-\frac 1p\right)_+$.
\begin{itemize}
\item Let $p\ge q$ or $p< q$, $\hat q\le 2$.
We set $\sigma(n)=\Psi(n^{-1/\theta} \varphi_{\theta,
\psi_\Lambda}^{-1}(n) )\varphi_{\theta,
\psi_\Lambda}^{-\kappa}(n)$. Then
$$
\vartheta_n({\bf S}_{p,u,w}({\cal A}), \, l_q({\cal A}))
\underset{\mathfrak{Z}}{\lesssim}
n^{-\frac{\kappa}{\theta}+\left(\frac 1q-\frac
1p\right)_+}\sigma(n).
$$
\item Let $p<q$, $\hat q>2$. We set $\theta_1=
\frac{\kappa}{\theta} + \min\left\{\frac 1p-\frac 1q, \, \frac
12-\frac{1}{\hat q}\right\}$, $\theta_2 = \frac{\hat
q\kappa}{2\theta}$, $\sigma_1(n) = \Psi(n^{-1/\theta}
\varphi_{\theta, \psi_\Lambda}^{-1}(n) )\varphi_{\theta,
\psi_\Lambda}^{-\kappa}(n)$, $\sigma_2(n)=\sigma_1(n^{\hat q/2})$.
Let $\theta_{j_*}=\min _{j=1,2} \theta_j$, $\theta_1\ne \theta_2$.
Then
$$
\vartheta_n({\bf S}_{p,u,w}({\cal A}), \, l_q({\cal A}))
\underset{\mathfrak{Z}}{\lesssim}
n^{-\theta_{j_*}}\sigma_{j_*}(n).
$$
\end{itemize}
\item Suppose that $\theta>0$ and $\kappa =-\theta\left(\frac 1q-\frac 1p\right)_+$.
Then $$\vartheta_n({\bf S}_{p,u,w}({\cal A}), \, l_q({\cal A}))
\underset{\mathfrak{Z}}{\lesssim} (\log
n)^{-\alpha+(1-\gamma)\left(\frac 1q-\frac 1p\right)_+} \rho(\log
n)\tau^{-\left(\frac 1q-\frac 1p\right)_+}(\log n).$$
\item Suppose that $\theta=0$ and $\kappa =0$.
\begin{itemize}
\item Let $p\ge q$ or $\hat q\le 2$. We set $\sigma(n)=
\rho\left(n^{\frac{1}{1-\gamma}} \varphi _{1-\gamma,
\tau^{-1}}(n)\right) \varphi_{1-\gamma,\tau^{-1}} ^{-\alpha}(n)$.
Then $$\vartheta_n({\bf S}_{p,u,w}({\cal A}), \, l_q({\cal A}))
\underset{\mathfrak{Z}}{\lesssim}
n^{-\frac{\alpha}{1-\gamma}+\left(\frac 1q-\frac
1p\right)_+}\sigma(n).$$
\item Let $p<q$, $\hat q>2$. We set $\theta_1=\frac{\alpha}{1-\gamma}
+\min\left\{\frac 1p-\frac 1q, \, \frac 12-\frac{1}{\hat
q}\right\}$, $\theta_2=\frac{\hat q\alpha}{2(1-\gamma)}$,
$\sigma_1(n)=
\rho\left(n^{\frac{1}{1-\gamma}}\varphi_{1-\gamma,\tau^{-1}}(n)\right)
\varphi_{1-\gamma,\tau^{-1}}^{-\alpha}(n)$, $\sigma_2(n)=
\sigma_1(n^{\hat q/2})$. Let $\theta_{j_*}=\min _{j=1,2}
\theta_j$, $\theta_1\ne \theta_2$. Then
$$
\vartheta_n({\bf S}_{p,u,w}({\cal A}), \, l_q({\cal A}))
\underset{\mathfrak{Z}_*}{\lesssim}
n^{-\theta_{j_*}}\sigma_{j_*}(n).
$$
\end{itemize}
\end{enumerate}
If there exists $\hat c>0$ such that ${\rm card}\, {\bf V}_k^{\cal
A}(\xi_0) \ge \frac{\hat c}{h(2^{-m_*k})}$ for any $k\in \Z_+$,
then similar lower estimates hold.
\end{Trm}
If $\theta=0$ and $\kappa >0$, then the widths can be estimated
from above by some decreasing geometric progression.

In conclusion, the author expresses her sincere gratitude to  A.R.
Alimov for help with translation.

\end{document}